\documentclass[titlepage,12pt]{article} 
\usepackage{hyperref}
\usepackage[usenames,dvipsnames]{pstricks} 
\usepackage{pst-plot} 
\usepackage[title]{appendix}
\usepackage{amssymb,amsthm,amsmath} 
\usepackage{mathrsfs}
\usepackage[a4paper]{geometry}
\usepackage{datetime2}
\usepackage[utf8]{inputenc}
\usepackage[italian,english]{babel}

\date{}

\newcommand{\ep}{\varepsilon}
\newcommand{\re}{\mathbb{R}}

\newcommand{\n}{\mathbb{N}}

\newcommand{\holder}{H\"older}
\newcommand{\PS}{\mathcal{PS}}
\newcommand{\bl}{b_{\lambda}}
\newcommand{\cl}{c_{\lambda}}
\newcommand{\vl}{u_{\lambda}}
\newcommand{\ul}{u_{\lambda}}

\newcommand{\PSp}{\mathcal{PS}}
\newcommand{\what}{\widehat{w}}
\newcommand{\ft}{\mathscr{F}}

\newtheorem{thm}{Theorem}[section]

\newtheorem{rmk}[thm]{Remark}
\newtheorem{prop}[thm]{Proposition}
\newtheorem{defn}[thm]{Definition}
\newtheorem{cor}[thm]{Corollary}

\newtheorem{lemma}[thm]{Lemma}

\title{Optimal derivative loss for abstract wave equations}

\author{Marina Ghisi\vspace{1ex}\\ 
{\normalsize Università degli Studi di Pisa} \\
{\normalsize Dipartimento di Matematica}\\ 
{\normalsize PISA (Italy)}\\
{\normalsize e-mail: \texttt{marina.ghisi@unipi.it}}
\and
Massimo Gobbino\vspace{1ex}\\ 
{\normalsize Università degli Studi di Pisa} \\
{\normalsize Dipartimento di Matematica}\\ 
{\normalsize PISA (Italy)}\\  
{\normalsize e-mail: \texttt{massimo.gobbino@unipi.it}}
}

\begin{document}
\maketitle

\begin{abstract}

We consider an abstract wave equation with a propagation speed that depends only on time. We assume that the propagation speed is differentiable for positive times, continuous up to the origin, but with first derivative that is potentially singular at the origin. 

We examine the derivative loss of solutions, and in particular we investigate which conditions on the modulus of continuity and on the behavior of the derivative in the origin yield, respectively, no derivative loss, an arbitrarily small derivative loss, a finite derivative loss, or an infinite derivative loss. As expected, we obtain that stronger assumptions on the modulus of continuity can compensate weaker assumptions on the growth of the derivative, and viceversa.

Suitable counterexamples show that our results are sharp. We prove indeed that, for every set of conditions, the class of propagation speeds that satisfy the given conditions, and for which the corresponding equation exhibits a derivative loss as large as possible, is nonempty and actually also residual in the sense of Baire category. 

\vspace{6ex}

\noindent{\bf Mathematics Subject Classification 2010 (MSC2010):} 
35L90 (35L20, 35B30, 35B65).

% 35L10: Second-order hyperbolic equations
% 35L20: Initial-boundary value problems for second-order hyperbolic equations
% 35L80: Degenerate hyperbolic equations
% 35L90: Abstract hyperbolic equations
% 35B30: Dependence of solutions on initial and boundary data, parameters
% 35B65: Smoothness and regularity of solutions
		
\vspace{6ex}


\end{abstract}

%%%%%%%%%%%%%%%%%%%%%
%                   %
%   Inizio lavoro   %
%                   %
%%%%%%%%%%%%%%%%%%%%%
 
\section{Introduction}

In this paper we consider the wave equation
\begin{equation}
u_{tt}-c(t)\Delta u=0,
\label{defn:CWE}
\end{equation}
and its abstract version
\begin{equation}
u''(t)+c(t)Au(t)=0,
\label{defn:AWE}
\end{equation}
where $A$ is a linear nonnegative self-adjoint operator with domain $D(A)$ in some real Hilbert space $H$. We always assume that the coefficient $c(t)$, which in the model (\ref{defn:CWE}) represents the square of the propagation speed, is defined in some time interval $(0,T_{0})$, and satisfies the \emph{strict hyperbolicity assumption}
\begin{equation}
0<\mu_{1}\leq c(t)\leq\mu_{2}
\qquad
\forall t\in(0,T_{0}).
\label{hp:c-sh}
\end{equation}

We investigate the regularity of solutions to (\ref{defn:AWE}) with initial data
\begin{equation}
u(0)=u_{0},
\qquad
u'(0)=u_{1}.
\label{eqn:data}
\end{equation}

We recall that problem (\ref{defn:AWE})--(\ref{eqn:data}) admits a unique solution for large classes of initial data, even if the coefficient $c(t)$ is just in $L^{1}((0,T_{0}))$, without sign conditions. Nevertheless, in general this solution is very weak, in the sense that it lives in a huge space of hyperdistributions, even if initial data are smooth.

Here we are interested in solutions with more ``space'' regularity. In order to state the definitions in the abstract setting we recall that, for every real number $\beta$, the operator $A^{\beta}$ is defined in a suitable domain $D(A^{\beta})$, which in the concrete case corresponds to the Sobolev space $H^{2\beta}$ (distributions if $\beta<0$). 

\begin{defn}[Well-posedness vs derivative loss]\label{defn:regularity}
\begin{em}
\mbox{} 

\begin{itemize}

\item  (No derivative loss). Problem (\ref{defn:AWE})--(\ref{eqn:data}) is said to be well-posed with \emph{no derivative loss} if, for every pair of initial data $(u_{0},u_{1})\in D(A^{\beta+1/2})\times D(A^{\beta})$, the unique solution satisfies
\begin{equation}
(u(t),u'(t))\in D(A^{\beta+1/2})\times D(A^{\beta})
\qquad
\forall t\in[0,T_{0}].
\nonumber
\end{equation}

\item  (Arbitrarily small derivative loss). Problem (\ref{defn:AWE})--(\ref{eqn:data}) is said to be well-posed with (at most an) \emph{arbitrarily small derivative loss} if, for every pair of initial data $(u_{0},u_{1})\in D(A^{\beta+1/2})\times D(A^{\beta})$, the unique solution satisfies
\begin{equation}
(u(t),u'(t))\in D(A^{\beta-\ep+1/2})\times D(A^{\beta-\ep})
\qquad
\forall t\in[0,T_{0}]
\quad
\forall\ep>0.
\nonumber
\end{equation}

The arbitrarily small derivative loss does actually happen if there exists a pair of initial data $(u_{0},u_{1})\in D(A^{\beta+1/2})\times D(A^{\beta})$ such that the unique solution satisfies
\begin{equation}
(u(t),u'(t))\not\in D(A^{\beta+1/2})\times D(A^{\beta})
\qquad
\forall t\in(0,T_{0}].
\nonumber
\end{equation}

\item  (Finite derivative loss). Problem (\ref{defn:AWE})--(\ref{eqn:data}) is said to be well-posed with (at most~a) \emph{finite derivative loss} if there exists a positive real number $\delta$ such that, for every pair of initial data $(u_{0},u_{1})\in D(A^{\beta+1/2})\times D(A^{\beta})$, the unique solution satisfies
\begin{equation}
(u(t),u'(t))\in D(A^{\beta-\delta+1/2})\times D(A^{\beta-\delta})
\qquad
\forall t\in[0,T_{0}].
\nonumber
\end{equation}

The finite derivative loss does actually happen if there exist a function $\delta:(0,T_{0}]\to(0,+\infty)$, and a pair of initial data $(u_{0},u_{1})\in D(A^{\beta+1/2})\times D(A^{\beta})$, such that the unique solution satisfies
\begin{equation}
(u(t),u'(t))\not\in D(A^{\beta-\delta(t)+1/2})\times D(A^{\beta-\delta(t)})
\qquad
\forall t\in(0,T_{0}].
\nonumber
\end{equation}

\item  (Infinite derivative loss). Problem (\ref{defn:AWE})--(\ref{eqn:data}) is said to exhibit and \emph{infinite derivative loss} if there exists a pair of initial data $(u_{0},u_{1})\in D(A^{\beta+1/2})\times D(A^{\beta})$, such that the unique solution satisfies
\begin{equation}
(u(t),u'(t))\not\in D(A^{-\gamma+1/2})\times D(A^{-\gamma})
\qquad
\forall\gamma>0,\quad\forall t\in(0,T_{0}].
\nonumber
\end{equation}

\end{itemize}

\end{em}
\end{defn}

Due to the linearity of the equation, all the definitions stated above do not depend on the choice of $\beta$. In words, no derivative loss means more generally that all solutions live in the same space of the initial data, while finite derivative loss means that the ``space regularity'' of the solution for positive times is less than the corresponding regularity of initial data. The parameter $\delta$ measures this loss of regularity, which is a true loss of derivatives in the concrete case where the domains of powers of $A$ are actually Sobolev spaces. The arbitrary small derivative loss is a condition in between no derivative loss and finite derivative loss: in this case solutions for positive times do not remain in the same space of initial data, but in all spaces with smaller exponents. Finally, the infinite derivative loss is a dramatic loss of regularity: in the concrete case it means the existence of solutions whose initial data have any given Sobolev regularity, and nevertheless they are not even distributions for positive times.

It is well known that the derivative loss of solutions depends on the time-regularity of the coefficient $c(t)$, and in particular on its oscillatory behavior. This regularity has been measured in different ways in the literature. Let us mention some of them.

\paragraph{\textmd{\textit{Modulus of continuity}}}

Let us assume that $c(t)$ is continuous in the closed interval $[0,T_{0}]$, and let $\omega:[0,+\infty)\to[0,+\infty)$ be a function such that
\begin{equation}
|c(t)-c(s)|\leq\omega(|t-s|)
\qquad
\forall(t,s)\in[0,T_{0}]^{2}.
\label{hp:c-omega}
\end{equation}

Any function $\omega$ with this property is called a \emph{modulus of continuity} for $c(t)$ in $[0,T_{0}]$.

The relations between the modulus of continuity of the coefficient and the regularity of solutions was investigated for the first time by F.~Colombini, E.~De~Giorgi and S.~Spagnolo in the seminal paper~\cite{dgcs}. The result was then refined and extended in many subsequent papers (see for example~\cite{1989-ENS-ColSpa,1995-Duke-ColLer,2006-JDE-CicCol,2015-CPDE-ColDSaFanMet}). Concerning the derivative loss of solutions, the situation is summarized in Table~\ref{table:omega}, where the assumptions in the first column refer to the behavior of $\omega(\sigma)$ as $\sigma\to 0^{+}$.
\begin{table}[h]
\centering 
\renewcommand{\arraystretch}{1.5}
\begin{tabular}{|c|l|}
\hline
$\omega(\sigma)\sim\sigma$
&
no derivative loss 
\\
\hline
$\sigma\ll\omega(\sigma)\ll\sigma|\log\sigma|$
&
arbitrarily small derivative loss
\\
\hline
$\omega(\sigma)\sim\sigma|\log\sigma|$
&
finite derivative loss
\\
\hline
$\omega(\sigma)\gg\sigma|\log\sigma|$
&
infinite derivative loss
\\
\hline
\end{tabular}
  
\caption{Modulus of continuity of $c(t)$ vs derivative loss}
\label{table:omega}
\end{table}

Now we know that all the results stated in Table~\ref{table:omega} are \emph{residually optimal}, namely for every modulus of continuity $\omega$ the set of coefficients $c(t)$ that are $\omega$-continuous, and for which problem (\ref{defn:AWE})--(\ref{eqn:data}) does exhibit the prescribed derivative loss is residual in the sense of Baire category (see~\cite{gg:residual,gg:DGCS-critical}).

\paragraph{\textmd{\textit{Singular behavior of the derivative at the origin}}}

Let us assume that $c(t)$ is differentiable for positive times, and let $\theta:(0,+\infty)\to(0,+\infty)$ be a nonincreasing function such that
\begin{equation}
|c'(t)|\leq\theta(t)
\qquad
\forall t\in(0,T_{0}].
\label{hp:c-theta}
\end{equation}

We point out that now $c(t)$ is not required to be continuous in $t=0$, and also $\theta(t)$ is allowed to diverge as $t\to 0^{+}$, and actually this is the interesting case. The effect of this singular behavior of $c'(t)$ in $t=0$ was studied by F.~Colombini, D.~Del Santo and T.~Kinoshita in~\cite{2002-SNS-ColDSaKin} in the case where $\theta(t)\sim 1/t^{\beta}$. Concerning the regularity of solutions, the situation is summarized in Table~\ref{table:theta}, where the assumptions in the first column refer to the behavior of $\theta(t)$ as $t\to 0^{+}$. Note that, due to the strict hyperbolicity condition, the divergence of the integral of $|c'(t)|$ implies a highly oscillatory behavior of $c(t)$.
\begin{table}[h]
\centering 
\renewcommand{\arraystretch}{2.3}
\begin{tabular}{|c|l|}
\hline
$\displaystyle\int_{0}^{T}\theta(t)\,dt<+\infty$
&
well-posedness in Sobolev spaces 
\\[1ex]
\hline
$\displaystyle\int_{0}^{T}\theta(t)\,dt=+\infty$\quad and\quad $\theta(t)\ll\dfrac{1}{t}$
&
arbitrarily small derivative loss
\\[1ex]
\hline
$\theta(t)\sim\dfrac{1}{t}$
&
finite derivative loss
\\[1ex]
\hline
$\theta(t)\gg\dfrac{1}{t}$
&
infinite derivative loss
\\[1ex]
\hline
\end{tabular}
  
\caption{Singular behavior of $c'(t)$ vs derivative loss}
\label{table:theta}
\end{table}

The optimality of many points in Table~\ref{table:theta} remained open for almost two decades. The last steps are contained in~\cite{gg:CDSR-optimal} and in the present paper.

\paragraph{\textmd{\textit{Singular behavior of the first two derivatives at the origin}}}

A natural way to extend the results of the previous paragraph is to consider the first two derivatives of the coefficient $c(t)$, with the hope that a bound on $|c'(t)|$ and $|c''(t)|$ can prevent $c(t)$ from oscillating too fast and yield a smaller derivative loss. A first result in this direction was obtained by T.~Yamazaki in~\cite{1990-CPDE-yamazaki}. The assumption is that $c(t)$ is twice differentiable for positive times and satisfies, up to multiplicative constants, the estimates
\begin{equation}
|c'(t)|\leq\frac{1}{t}
\qquad\quad\text{and}\quad\qquad
|c''(t)|\leq\frac{1}{t^{2}}
\nonumber
\end{equation}
for every $t\in(0,T_{0}]$. Under these assumptions she proved that problem (\ref{defn:AWE})--(\ref{eqn:data}) is well-posed with no derivative loss.

Some years later, F.~Colombini, D.~Del Santo and M.~Reissig in~\cite{2003-BSM-ColDSaRei} assumed that, up to multiplicative constants, the coefficient $c(t)$ satisfies
\begin{equation}
|c'(t)|\leq \frac{|\log t|}{t}
\qquad\quad\text{and}\quad\qquad
|c''(t)|\leq \left(\frac{\log t}{t}\right)^{2}
\nonumber
\end{equation}
in a right neighborhood of the origin. Under these assumptions they proved that problem (\ref{defn:AWE})--(\ref{eqn:data}) is well-posed with finite derivative loss (see also~\cite{2003-MMAS-Hirosawa,2003-MatN-Hirosawa}). 

These results were extended and unified recently in~\cite{gg:CDSR-optimal}, were it is assumed that
\begin{equation}
|c'(t)|\leq\frac{\varphi(t)}{t}
\qquad\quad\text{and}\quad\qquad
|c''(t)|\leq\left(\frac{\varphi(t)}{t}\right)^{2}\exp(\psi(t)),
\nonumber
\end{equation}
where $\varphi:(0,T_{0})\to(0,+\infty)$ and $\psi:(0,T_{0})\to(0,+\infty)$ are suitable nonincreasing and continuous functions. The results of~\cite{gg:CDSR-optimal} are summarized in Table~\ref{table:c'-c''}, where the first column refers to the behavior as $t\to 0^{+}$.
\begin{table}[h]
\centering 
\renewcommand{\arraystretch}{1.5}
\begin{tabular}{|c|l|}
\hline
$(1+\varphi(t))\psi(t)\sim 1$
&
no derivative loss 
\\
\hline
$1\ll(1+\varphi(t))\psi(t)\ll|\log t|$
&
arbitrarily small derivative loss
\\
\hline
$(1+\varphi(t))\psi(t)\sim|\log t|$
&
finite derivative loss
\\
\hline
$(1+\varphi(t))\psi(t)\gg|\log t|$
&
infinite derivative loss
\\
\hline
\end{tabular}
  
\caption{Singular behavior of $c'(t)$ and $c''(t)$ vs derivative loss}
\label{table:c'-c''}
\end{table}

We observe that in the case where $\varphi(t)$ and $\psi(t)$ are constant functions this is exactly the result of~\cite{1990-CPDE-yamazaki}, while in the case where $\varphi(t)\sim|\log t|$ and $\psi(t)$ is constant this is exactly the result of~\cite{2003-BSM-ColDSaRei}. Again, the derivative loss prescribed by Table~\ref{table:c'-c''} is residually optimal.

\paragraph{\textmd{\textit{Modulus of continuity and first derivatives}}}

In this paper we combine the assumptions on the modulus of continuity and on the first derivative. More precisely, we assume that $c(t)$ is continuous in the closed interval $[0,T_{0}]$, differentiable in the half-open interval $(0,T_{0}]$, and that it satisfies both (\ref{hp:c-omega}) and (\ref{hp:c-theta}) for suitable functions $\omega$ and $\theta$. The case where $\omega(\sigma)\sim\sigma^{\alpha}$ and $\theta(t)\sim1/t^{\beta}$ was considered by F.~Colombini, D.~Del Santo and T.~Kinoshita in~\cite{2002-SNS-ColDSaKin} (see also~\cite{2003-DIE-CicCol}), while more subtle examples where considered by F.~Colombini, D.~Del santo and M.~Reissig in~\cite{2003-BSM-ColDSaRei} and by D.~Del Santo, T.~Kinoshita and M.~Reissig in~\cite{2007-DIE-DSaKinRei} (see also~\cite{2005-ADE-KinRei}). Here we unify and improve some of their results, both on the positive and on the negative side. More important, we show that all those special examples fit into a common framework.

Our main result is that the key quantity
\begin{equation}
m(\lambda):=\min\left\{\lambda\,\omega\left(\frac{1}{\lambda}\right)s+
\int_{s}^{T_{0}}\theta(t)\,dt:s\in[0,T_{0}]\right\}
\qquad
\forall\lambda>0
\label{defn:m-lambda}
\end{equation}
determines the derivative loss of solutions to problem (\ref{defn:AWE})--(\ref{eqn:data}) according to Table~\ref{table:m-lambda}, where the first column refers to the behavior of $m(\lambda)$ as $\lambda\to+\infty$.
\begin{table}[h]
\centering 
\renewcommand{\arraystretch}{1.5}
\begin{tabular}{|c|l|}
\hline
$m(\lambda)\sim 1$
&
well-posedness in Sobolev spaces 
\\
\hline
$1\ll m(\lambda)\ll\log\lambda$
&
arbitrarily small derivative loss
\\
\hline
$m(\lambda)\sim\log\lambda$
&
finite derivative loss
\\
\hline
$m(\lambda)\gg\log\lambda$
&
infinite derivative loss
\\
\hline
\end{tabular}
  
\caption{Modulus of continuity and singular behavior of $c'(t)$ vs derivative loss}
\label{table:m-lambda}
\end{table}

As usual, all results are residually optimal.

\paragraph{\textmd{\textit{Overview of the technique -- Upper bound for the derivative loss}}}

From the technical point of view, it is well-known that the spectral theorem for self-adjoint nonnegative operators reduces the abstract equation (\ref{defn:AWE}) to the family of ordinary differential equations
\begin{equation}
u_{\lambda}''(t)+\lambda^{2}c(t)u_{\lambda}(t)=0,
\label{eqn:u-lambda}
\end{equation}
where $\lambda$ is a positive real parameter. In particular, if one can prove that solutions to
(\ref{eqn:u-lambda}) satisfy an estimate of the form
\begin{equation}
\vl'(t)^{2}+\lambda^{2}\vl(t)^{2}\leq
\left(\vl'(0)^{2}+\lambda^{2}\vl(0)^{2}\right)\exp(\phi_{+}(\lambda,t))
\qquad
\forall t\in[0,T_{0}],
\label{est:energy-phi}
\end{equation}
where $\phi_{+}(\lambda,t)$ is a function independent of initial data, then the behavior of $\phi_{+}(\lambda,t)$ as $\lambda\to +\infty$ determines the maximum possible derivative loss of solutions to (\ref{defn:AWE})--(\ref{eqn:data}) according to Table~\ref{table:phi}.
\begin{table}[h]
\centering 
\renewcommand{\arraystretch}{1.5}
\begin{tabular}{|c|l|}
\hline
$\phi_{+}(\lambda,t)\sim 1$
&
well-posedness in Sobolev spaces 
\\
\hline
$1\ll\phi_{+}(\lambda,t)\ll\log\lambda$
&
arbitrarily small derivative loss
\\
\hline
$\phi_{+}(\lambda,t)\sim\log\lambda$
&
finite derivative loss
\\
\hline
$\phi_{+}(\lambda,t)\gg\log\lambda$
&
infinite derivative loss
\\
\hline
\end{tabular}
  
\caption{Energy growth for solutions to (\ref{eqn:u-lambda}) vs derivative loss}
\label{table:phi}
\end{table}

When $c(t)$ is $\omega$-continuous, the approximated energy estimates introduced in~\cite{dgcs} allow to show that (\ref{est:energy-phi}) holds true with
\begin{equation}
\phi_{+}(\lambda,t)\sim\lambda\,\omega\left(\frac{1}{\lambda}\right)t,
\nonumber
%\label{defn:phi-omega}
\end{equation}
and this explains all the results of Table~\ref{table:omega} and much more, for example the well-posedness  in Gevrey spaces in the case of \holder\ continuous coefficients. 

In a different direction, when $c(t)$ is of class $C^{1}$ in the closed interval $[0,T_{0}]$, the classical hyperbolic estimates give that (\ref{est:energy-phi}) holds true with (note that in this case there is no dependence on $\lambda$)
\begin{equation}
\phi_{+}(\lambda,t)\sim\int_{0}^{t}|c'(\tau)|\,d\tau.
%\label{defn:phi-omega}
\nonumber
\end{equation}

In this paper $|c'(t)|$ is not necessarily integrable in a right neighborhood of the origin, and therefore we need to mix the two techniques, in the sense that we use an estimate of the first type in some initial interval $[0,s]$, followed by an estimate of the second type in the remaining interval $[s,t]$, where (\ref{hp:c-theta}) provides a control on the derivative. When we optimize with respect to $s$ we conclude that now (\ref{est:energy-phi}) holds true with $\phi_{+}(\lambda,t)\sim m(\lambda)$, with $m(\lambda)$ given by (\ref{defn:m-lambda}). This is enough to conclude that the derivative loss of solutions to (\ref{defn:AWE}) is at most the one given in Table~\ref{table:m-lambda}. We refer to statement~(1) of Theorem~\ref{thm:main} for the details.

\paragraph{\textmd{\textit{Overview of the technique -- Road map to counterexamples}}}

The main contribution of this paper is the construction of solutions that exhibit a prescribed derivative loss. This is a much more delicate issue, since it requires to show that estimates of the form (\ref{est:energy-phi}) are in some sense optimal. In statement~(2) of Theorem~\ref{thm:main} the idea is to look for coefficients $c(t)$ such that solutions to (\ref{eqn:u-lambda}) satisfy
\begin{equation}
\vl'(t)^{2}+\lambda^{2}\vl(t)^{2}\geq
\left(\vl'(0)^{2}+\lambda^{2}\vl(0)^{2}\right)\exp(\phi_{-}(\lambda,t))
\qquad
\forall t\in[0,T_{0}],
\label{est:energy-phi-below}
\end{equation}
at least on a sequence $\lambda_{n}\to+\infty$, with $\phi_{-}(\lambda,t)\sim m(\lambda)$ for positive times. These coefficients are called ``universal activators'' because the same coefficient induces the exponential growth of a sequence of solutions. They are the fundamental tool in the construction of counterexamples, as we show in Proposition~\ref{prop:ua2idl} for operators that admit an unbounded sequence of eigenvalues, and in Proposition~\ref{prop:ua2dl-gen} for general unbounded self-adjoint operators, which is the case, for example, of the concrete wave equation (\ref{defn:CWE}) on the whole space or an external domain.

Following the path introduced in~\cite{gg:residual,gg:DGCS-critical,gg:CDSR-optimal}, the existence of universal activators is reduced to the existence of families of ``asymptotic activators'', namely families $\{\cl(t)\}$ of coefficients such that solutions to 
\begin{equation}
u_{\lambda}''(t)+\lambda^{2}\cl(t)u_{\lambda}(t)=0
\label{eqn:u-c-lambda}
\end{equation}
satisfy (\ref{est:energy-phi-below}) when $\lambda$ is large enough.

We stress the difference between universal activators, where the same coefficient produces an exponential growth for a sequence of solutions, and asymptotic activators that achieve the same exponential growth by choosing a different coefficient for different values of $\lambda$ (note that in (\ref{eqn:u-c-lambda}) the coefficient does depend on $\lambda$). 

The main point proved in~\cite{gg:CDSR-optimal} is that the existence of sufficiently many families of asymptotic activators, within a certain class of coefficients, implies the existence of a residual set of universal activators in the same class. This is some sort of \emph{``nonlinear uniform boundedness principle''} (where nonlinear refers to the map coefficient $\mapsto$ solutions): if there exist sufficiently many families of objects that show \emph{asymptotically} the optimality of some estimate, then there are residually many objects that show \emph{directly} the optimality of the same estimate. Equivalently, if we can not estimate the norm of $(u(t),u'(t))$ is some space $Y$ in terms of the norm of $(u_{0},u_{1})$ in some space $X$, then there exist solutions such that $(u_{0},u_{1})$ lies in the space $X$, but $(u(t),u'(t))$ does not lie in the space $Y$ for positive times.

Finally, asymptotic activators are produced starting from the usual building blocks, introduced for the first time in~\cite{dgcs}, and then modified and adapted in the subsequent literature. The key observation is that
\begin{equation}
u_{\lambda}(t):=
\frac{1}{\lambda}\sin(\lambda t)
\exp\left(\frac{1}{8}\int_{0}^{t}\ep(s)\sin^{2}(\lambda s)\,ds\right)
\label{defn:ul-biblio}
\end{equation}
grows exponentially and solves (\ref{eqn:u-c-lambda}) with
\begin{equation}
\cl(t):=1-\frac{\ep(t)}{4\lambda}\sin(2\lambda t)
-\frac{\ep'(t)}{8\lambda^{2}}\sin^{2}(\lambda t)
-\frac{\ep(t)^{2}}{64\lambda^{2}}\sin^{4}(\lambda t).
\label{defn:cl-biblio}
\end{equation}

When showing the optimality of the results of Table~\ref{table:omega}, it is enough to choose $\ep(t)$ to be independent of $t$ and equal to $\lambda\,\omega(1/\lambda)$, up to multiplicative constants. In this way the integral in the exponential term of (\ref{defn:ul-biblio}) grows as a multiple of $\lambda\,\omega(1/\lambda)$, as required, and it is possible to control the modulus of continuity of the coefficient because $\cl(t)$ makes oscillations of order $\omega(1/\lambda)$ in intervals with length of order $1/\lambda$ (namely the period of the trigonometric terms). 

In this paper we consider the minimizer $s_{\lambda}$ in the minimum problem (\ref{defn:m-lambda}), and we check which of the two summands is bigger for $s=s_{\lambda}$. When it is the first one, we define again $\ep(t)$ as $\lambda\,\omega(1/\lambda)$ (see Proposition~\ref{prop:omega}). When it is the second one, namely the integral, one would like to choose $\ep(t)=\theta(t)$, so that the integral in the exponential term of (\ref{defn:ul-biblio}) grows as the integral of $\theta(t)$. This choice has several disadvantages, mainly because we need to control $\cl'(t)$, and therefore the presence of $\ep'(t)$ in (\ref{defn:cl-biblio}) forces to assume that $\theta(t)$ is twice differentiable, with a lot of control on its derivatives.

In Proposition~\ref{prop:theta} we overcome this difficulty by choosing $\ep(t)$ equal to a piecewise constant approximation of $\theta(t)$, changing the constant whenever the trigonometric terms vanish. In this way $\cl(t)$ remains Lipschitz continuous, the term with $\ep'(t)$ disappears, and the integral in (\ref{defn:ul-biblio}) is equivalent to a Riemann sum for the integral of $\theta(t)$. The lack of this type of construction, which becomes fundamental when the growth of $\theta(t)$ is close enough to $1/t$, is probably the reason why the previous results in the literature were not optimal.

\paragraph{\textmd{\textit{Related problems and future perspectives}}}

We hope that the methods of this paper, more than the results themselves, could be useful to deal with analogous problems. For sure the nonlinear uniform boundedness principle, namely the general path from asymptotic to universal activators, can be used to show in an efficient way the optimality of many positive results. Indeed, some of the known counterexamples are still stated in the form of impossibility of a certain energy estimate, in the spirit of asymptotic activators, and not as examples of solutions that actually lose derivatives (see for example \cite[Theorem~2.3 and Theorem~2.5]{2006-JDE-CicCol}, or \cite[Theorem~2.7]{2015-JMAA-EbeFitHir}).

In a different direction, we are confident that our techniques could shed some light also on a related problem studied in the last two decades in a series of papers by M.~Reissig and J.~Smith~\cite{2005-HokkMJ-ReiSmi}, F.~Colombini~\cite{2006-JDE-Colombini}, F.~Hirosawa~\cite{2007-MathAnn-Hirosawa,2021-JMAA-Hirosawa}, and M.~Ebert, L.~Fitriana and F.~Hirosawa~\cite{2015-JMAA-EbeFitHir}. They consider the wave equation (\ref{defn:CWE}) with a smooth coefficient $c(t)$ defined for all positive times, with two types of assumptions: the decay of some derivatives of $c(t)$ as $t\to+\infty$, and a ``stabilization condition'', namely some integral control on $|c(t)-c_{\infty}|$, where $c_{\infty}$ is a suitable constant. They are interested in what they call ``generalized energy conservation'', namely the boundedness of the ratio between the energy at time~$t$ and the energy at time~0. There are still some gaps between the positive results and the counterexamples (see for example~\cite[Table~1 and Table~2]{2015-JMAA-EbeFitHir}). The analogy with this paper is plausible: the decay of derivatives at infinity should correspond to the blow-up at the origin, the stabilization condition could correspond to the modulus of continuity, and the general energy conservation to well-posedness with no derivative loss. 

\paragraph{\textmd{\textit{Structure of the paper}}}

This paper is organized as follows. In Section~\ref{sec:statements} we state our main results and some consequences, and we comment on them. In Section~\ref{sec:well-posed} we prove the positive part, namely the energy estimates from above that yield a bound from above for the derivative loss. In Section~\ref{sec:counterexamples} we present the construction of the asymptotic activators, and how they lead to our counterexamples. Finally, in Section~\ref{sec:cor} we prove two corollaries concerning two special cases.

%\clearpage  

\setcounter{equation}{0}
\section{Statements}\label{sec:statements}

\subsection{Notations and main result}

Let us start by introducing some terminology and some notations.

\begin{defn}[Modulus of continuity]
\begin{em}

A \emph{modulus of continuity} is a function $\omega:[0,+\infty)\to[0,+\infty)$ such that
\begin{itemize}

\item  $\omega(\sigma)>0$ for $\sigma>0$, and $\omega(\sigma)\to 0$ as $\sigma\to 0^{+}$,

\item  the function $\sigma\mapsto\omega(\sigma)$ is nondecreasing,

\item  the function $\sigma\mapsto\sigma/\omega(\sigma)$ is nondecreasing.

\end{itemize}

A function $c:[0,T_{0}]\to\re$ is called $\omega$-continuous if it satisfies (\ref{hp:c-omega}) for some modulus of continuity $\omega$.

\end{em}
\end{defn}

\begin{defn}[Classes of coefficients]\label{defn:PS}
\begin{em}

Let $T_{0}$, $\mu_{1}$, $\mu_{2}$ be positive real numbers with $\mu_{2}>\mu_{1}$. Let $\omega$ be a modulus of continuity, and let $\theta:(0,T_{0})\to(0,+\infty)$ be a continuous and nonincreasing function.

We call $\PS(T_{0},\mu_{1},\mu_{2},\omega,\theta)$ the set of functions $c\in W^{\infty}_{loc}((0,T_{0}))$ that satisfy (\ref{hp:c-sh}) and (\ref{hp:c-omega}) in the pointwise sense, and (\ref{hp:c-theta}) in the almost everywhere sense.

\end{em}
\end{defn}

We observe that $\PS(T_{0},\mu_{1},\mu_{2},\omega,\theta)$ is a complete metric space with respect to the distance induced by the norm of $L^{\infty}((0,T_{0}))$. We observe also that the elements of this space are continuous in $[0,T_{0}]$, and therefore pointwise values $c(t)$ are well defined.

We are now ready to state our main result.

\begin{thm}[Main energy estimates]\label{thm:main}

Let $T_{0}$, $\mu_{1}$, $\mu_{2}$ be positive real numbers with $\mu_{2}>\mu_{1}$. Let $\omega$ be a modulus of continuity, and let $\theta:(0,T_{0})\to(0,+\infty)$ be a continuous and nonincreasing function. For every positive real number $\lambda$, let us define $m(\lambda)$ as in (\ref{defn:m-lambda}). Let us consider the classes of coefficients introduced in Definition~\ref{defn:PS}, and let us introduce the constants
\begin{equation}
M_{1}:=\left(\frac{\max\{1,\mu_{2}\}}{\min\{1,\mu_{1}\}}\right)^{2},
\qquad
M_{2}:=\frac{1}{\mu_{1}}+\frac{1}{\sqrt{\mu_{1}}},
\qquad
M_{3}:=\frac{1}{128\mu_{2}}\min\left\{1,\frac{\sqrt{\mu_{1}}}{\pi}\right\}.
\label{defn:M12}
\end{equation}

Then the following statements hold true.

\begin{enumerate}
\renewcommand{\labelenumi}{(\arabic{enumi})}

\item  \emph{(Energy estimate from above).} For every coefficient $c\in\PS(T_{0},\mu_{1},\mu_{2},\omega,\theta)$ and every positive real number $\lambda$ it turns out that every solution to (\ref{eqn:u-lambda}) satisfies
\begin{equation}
\ul'(t)^{2}+\lambda^{2}\ul(t)^{2}\leq
M_{1}\left(\ul'(0)^{2}+\lambda^{2}\ul(0)^{2}\right)\exp(M_{2}\,m(\lambda))
\qquad
\forall t\in[0,T_{0}].
\label{th:main-above}
\end{equation}

\item  \emph{(Energy estimate from below).} Let us assume that $m(\lambda)\to +\infty$ as $\lambda\to +\infty$.

Then, for every sequence of positive real numbers $\lambda_{n}\to +\infty$, the set of coefficients $c\in\PS(T_{0},\mu_{1},\mu_{2},\omega,\theta)$ such that the solutions to (\ref{eqn:u-lambda}) with initial data $\ul(0)=0$, $\ul'(0)=1$ satisfy
\begin{equation}
\limsup_{n\to +\infty}\left(|u_{\lambda_{n}}'(t)|^{2}+\lambda_{n}^{2}|u_{\lambda_{n}}(t)|^{2}\right)
\exp\left(-M_{3}\,m(\lambda_{n})\strut\right)\geq 1
\qquad
\forall t\in(0,T_{0}]
\nonumber
\end{equation}
is residual. 

\end{enumerate}

\end{thm}

The energy estimates of Theorem~\ref{thm:main} can be applied to the abstract wave equation, yielding the following result.

\begin{thm}[Derivative loss for the abstract wave equation]\label{thm:der-loss}

Let $T_{0}$, $\mu_{1}$, $\mu_{2}$, $\omega$, $\theta$, and $m(\lambda)$ be as in Theorem~\ref{thm:main}. Let us consider the classes of coefficients introduced in Definition~\ref{defn:PS}. Let us consider the abstract equation (\ref{defn:AWE}), where $A$ is a linear nonnegative self-adjoint operator in some real Hilbert space $H$.

Then the following statements hold true.
\begin{enumerate}
\renewcommand{\labelenumi}{(\arabic{enumi})}

\item  \emph{(Estimate from above for the derivative loss).} For every $c\in\PS(T_{0},\mu_{1},\mu_{2},\omega,\theta)$ the derivative loss of solutions to problem (\ref{defn:AWE})--(\ref{eqn:data}) is at most the one prescribed by Table~\ref{table:m-lambda}.

\item  \emph{(Optimality of the derivative loss).} If the operator $A$ is unbounded, then the set of coefficients $c\in\PS(T_{0},\mu_{1},\mu_{2},\omega,\theta)$ for which problem (\ref{defn:AWE})--(\ref{eqn:data}) exhibits exactly the derivative loss prescribed by Table~\ref{table:m-lambda} is residual.

\end{enumerate}

\end{thm}

\begin{rmk}
\begin{em}

One could define the class $\PS(T_{0},\mu_{1},\mu_{2},\omega,\theta)$ also by considering only  coefficients $c(t)$ that are of class $C^{1}$ for positive times, so that now (\ref{hp:c-theta}) can be asked in the pointwise sense. In this case a structure of complete metric space is induced by the norm 
\begin{equation}
\|c\|_{\theta}:=\max\{|c(t)|:t\in[0,T_{0}]\}+\sup\left\{\frac{|c'(t)|}{\theta(t)}:t\in(0,T_{0})\right\}.
\nonumber
\end{equation}

All the previous results hold true also in this restricted class. This is actually the approach that was carried on in~\cite{gg:CDSR-optimal}, and it delivers a residual (in this new space) class of counterexamples that are of class $C^{1}$ for positive times. On  the other hand, as explained in~\cite[section~4.5]{gg:CDSR-optimal}, it is always possible to produce counterexamples of class $C^{\infty}$. 

\end{em}
\end{rmk}

\subsection{Some examples}

Let us discuss the consequences of Theorem~\ref{thm:der-loss} in some special cases. A first result is that well-posedness with no derivative loss holds true in the class of coefficients $\PS(T_{0},\mu_{1},\mu_{2},\omega,\theta)$ if and only if either $\theta(t)$ guarantees that $c(t)$ has bounded variation, or $\omega(\sigma)$ guarantees that $c(t)$ is Lipschitz continuous.

\begin{cor}\label{cor:int-infty}

Let us consider the same setting of Theorem~\ref{thm:der-loss}. 

Let us assume that the operator $A$ is unbounded, and that
\begin{equation}
\lim_{\sigma\to 0^{+}}\frac{\sigma}{\omega(\sigma)}=0
\qquad\quad\text{and}\quad\qquad
\int_{0}^{T_{0}}\theta(t)\,dt=+\infty.
\label{hp:cor-int-infty}
\end{equation}

Then the set of coefficients $c\in\PSp(T_{0},\mu_{1},\mu_{2},\theta,\omega)$ for which problem (\ref{defn:AWE})--(\ref{eqn:data}) exhibits at least an arbitrarily small derivative loss is residual.

\end{cor}

Let us examine now the case where $\theta(t)\sim 1/t$. According to Table~\ref{table:theta} this assumption guarantees that the derivative loss is at most finite. Now we can show that the derivative loss is actually arbitrarily small if $c(t)$ is $\alpha$-\holder\ continuous for every $\alpha\in(0,1)$, and this assumption is optimal.

\begin{cor}\label{cor:1/t}

Let us consider the same setting of Theorem~\ref{thm:der-loss}.

\begin{enumerate}
\renewcommand{\labelenumi}{(\arabic{enumi})}

\item  \emph{(Arbitrarily small derivative loss).} Let us assume that the modulus of continuity satisfies
\begin{equation}
\forall\alpha\in(0,1)
\qquad
\lim_{\sigma\to 0^{+}}\frac{\omega(\sigma)}{\sigma^{\alpha}}=0,
\label{hp:cor-1/t-omega-1}
\end{equation}
and the function $\theta(t)$ satisfies
\begin{equation}
\limsup_{t\to 0^{+}}t\cdot\theta(t)<+\infty.
\label{hp:cor-1/t-theta-1}
\end{equation}

Then, for every propagation speed $c\in\PSp(T_{0},\mu_{1},\mu_{2},\theta,\omega)$, problem (\ref{defn:AWE})--(\ref{eqn:data}) has at most an arbitrarily small derivative loss.

\item  \emph{(Finite derivative loss).} Let us assume that the operator $A$ is unbounded, that the modulus of continuity satisfies
\begin{equation}
\exists\alpha\in(0,1)
\qquad
\liminf_{\sigma\to 0^{+}}\frac{\omega(\sigma)}{\sigma^{\alpha}}>0,
\label{hp:cor-1/t-omega-2}
\end{equation}
and that the function $\theta(t)$ satisfies
\begin{equation}
\liminf_{t\to 0^{+}}t\cdot\theta(t)>0.
\label{hp:cor-1/t-theta-2}
\end{equation}

Then the set of propagation speeds $c\in\PSp(T_{0},\mu_{1},\mu_{2},\theta,\omega)$ for which problem (\ref{defn:AWE})--(\ref{eqn:data}) has at least a finite derivative loss is residual.

\end{enumerate}

\end{cor}

Finally, let us examine the case where $\theta(t)\gg 1/t$. In this case problem (\ref{defn:AWE})--(\ref{eqn:data}) can exhibit any type of derivative loss, depending on the modulus of continuity $\omega$. The modulus of continuity that guarantees a finite derivative loss is always stronger than any \holder\ modulus $\sigma^{\alpha}$ with $\alpha\in(0,1)$, but weaker that the classical $\sigma|\log\sigma|$ that guarantees a finite derivative loss even without any assumption on $c'(t)$. In Table~\ref{table:theta-omega} we display, for some special choices of $\theta(t)$, the moduli of continuity that guarantee that $m(\lambda)\sim\log\lambda$, and hence a finite derivative loss. As usual, these choices of $\omega(\sigma)$ represent the threshold between the arbitrary small derivative loss and the infinite derivative loss. 

\begin{table}[h]
\centering 
\renewcommand{\arraystretch}{2.3}
\begin{tabular}{|c|c|}
\hline
$\theta(t)\sim\dfrac{|\log t|}{t}$
&
$\exists R>0\quad\omega(\sigma)\sim\sigma\exp\left(R\,|\log\sigma|^{1/2}\right)$ 
\\[1ex]
\hline
$\theta(t)\sim\dfrac{1}{t^{1+\beta}}\quad(\beta>0)$
&
$\omega(\sigma)\sim\sigma|\log\sigma|^{1+1/\beta}$
\\[1ex]
\hline
$\theta(t)\sim\exp(1/t)$
&
$\omega(\sigma)\sim\sigma|\log\sigma|\cdot\log(|\log\sigma|)$
\\[1ex]
\hline
$\theta(t)\sim\dfrac{1}{t^{\beta}}\exp(1/t)\quad(\beta>0)$
&
$\omega(\sigma)\sim\sigma|\log\sigma|\cdot\log(|\log\sigma|)$
\\[1ex]
\hline
\end{tabular}
  
\caption{Examples of finite derivative loss when $\theta(t)\gg 1/t$}
\label{table:theta-omega}
\end{table}

\begin{rmk}
\begin{em}

Some of the choices of Table~\ref{table:theta-omega} are considered in the previous literature, but without obtaining the optimal result. For example, the modulus $\omega(\sigma)$ of the first row is considered in~\cite[Section~4.5]{2005-ADE-KinRei}, where a finite derivative loss is obtained with the correct condition $\theta(t)\sim|\log t|/t$, and in~\cite[Example~1.2(ii)]{2003-BSM-ColDSaRei}, where an infinite derivative loss is obtained, but only with the stronger condition $\theta(t)\gg|\log t|^{2}/t$. The modulus $\omega(\sigma)$ of the second row is considered in~\cite[Example~1.2(i)]{2003-BSM-ColDSaRei}, where an infinite derivative loss is obtained, but only with the stronger condition $\theta(t)\gg1/t^{1+\beta}\cdot\log^{1+\beta}(|\log t|)$. 

Similarly, the choices of $\theta(t)$ of the second and the fourth row (the latter in the special case $\beta=2$) are considered in~\cite[Example~2.1 and Example~2.2]{2007-DIE-DSaKinRei}, and in both cases a finite derivative loss is obtained with a stronger modulus of continuity, and there is no mention of infinite and arbitrarily small derivative loss.

\end{em}
\end{rmk}

\subsection{Comments}

We conclude by speculating on our main results.

\begin{rmk}[Limit cases]\label{rmk:limit-cases}
\begin{em}

The cases where one prescribes only the modulus of continuity $\omega(\sigma)$, or only the blow-up rate $\theta(t)$ of the derivative, can be included in Theorem~\ref{thm:main} and Theorem~\ref{thm:der-loss} as special limit cases.

If we prescribe only the modulus of continuity $\omega(\sigma)$, we can think that $\theta(t)\equiv +\infty$, and therefore imagine that in (\ref{defn:m-lambda}) the minimum is attained when $s=T_{0}$. We obtain that $m(\lambda)=\lambda\,\omega(1/\lambda)T_{0}$, which explains the results of Table~\ref{table:omega}.

If we prescribe only $\theta(t)$, and of course also the strict hyperbolicity condition, we can think  to extend the notion of modulus of continuity in order to include the border-line case in which $\omega(\sigma)\equiv\mu_{2}-\mu_{1}$, and then define
\begin{equation}
m(\lambda)=\min\left\{(\mu_{2}-\mu_{1})\lambda\,s+\int_{s}^{T_{0}}\theta(t)\,dt:s\in[0,T_{0}]\right\}.
\nonumber
\end{equation}

The minimum is attained when $\theta(s)=(\mu_{2}-\mu_{1})\lambda$, and with some standard calculus we obtain all the results of Table~\ref{table:theta}.

\end{em}
\end{rmk}

\begin{rmk}[Quantitative estimate of the derivative loss]
\begin{em}

In the cases where $m(\lambda)\sim\log\lambda$, and more precisely
\begin{equation}
0<\liminf_{\lambda\to +\infty}\frac{m(\lambda)}{\log\lambda}\leq
\limsup_{\lambda\to +\infty}\frac{m(\lambda)}{\log\lambda}<+\infty,
\nonumber
\end{equation}
the liminf and limsup above provide, respectively, an estimate from below and from above for the finite derivative loss, namely for the constant $\delta$ that appears in the definition. 

\end{em}
\end{rmk}

\begin{rmk}[Progressive vs instantaneous derivative loss]
\begin{em}

There is a subtle difference in the derivative loss between the case where only the modulus of continuity is prescribed, and the cases where we assume $c(t)$ to be differentiable for positive times. For the sake of simplicity, let us limit ourselves to the finite derivative loss.

In the case where one prescribes only the modulus of continuity $\omega(\sigma)\sim\sigma|\log\sigma|$, the finite derivative loss is in general \emph{progressive} in the sense that, if $(u_{0},u_{1})\in D(A^{\beta+1/2})\times D(A^{\beta})$ for some $\beta$, then $(u(t),u'(t))\in D(A^{\beta-\delta t+1/2})\times D(A^{\beta-\delta t})$ for positive times. In words, this means that the derivative loss increases with time, and tends to~0 as $t\to 0^{+}$.

When $c(t)$ is of class $C^{1}$ for positive times, then any form of derivative loss is \emph{instantaneous}, and in particular a finite derivative does not tend to~0 as $t\to 0^{+}$ (but of course now $\omega(\sigma)\gg\sigma|\log\sigma|$). After the initial loss of regularity, in this case there is no further loss of derivatives, in the sense that the implication
\begin{equation}
(u(t_{0}),u'(t_{0}))\in D(A^{\gamma+1/2})\times D(A^{\gamma})
\quad\Longrightarrow\quad
(u(t),u'(t))\in D(A^{\gamma+1/2})\times D(A^{\gamma})
\nonumber
\end{equation}
holds true for every $0<t_{0}\leq t\leq T_{0}$ and every $\gamma$. In words, this means that the singular behavior of $c(t)$ in the origin is responsible for the instantaneous loss of derivatives but, after the initial loss, the regularity of solutions is preserved by the smoothness of $c(t)$.

\end{em}
\end{rmk}

\begin{rmk}[Well-posedness in Gevrey spaces]
\begin{em}

The consequences of Theorem~\ref{thm:main} go far beyond Theorem~\ref{thm:der-loss}, and in particular beyond the classification of derivative loss according to Table~\ref{table:m-lambda}. Indeed, the behavior of $m(\lambda)$ as $\lambda\to +\infty$ provides a sharp ``measure'' of the derivative loss, even when it is infinite or arbitrarily small. The formalization of this idea relies on the notion of generalized Gevrey spaces, or Gevrey distributions (we refer to~\cite[Definition~2.2]{gg:DGCS-critical} for more details on the abstract functional setting for abstract wave equations). Just to give some examples, let us stick to standard Gevrey spaces. The positive side is represented by well-posedness results, as follows.
\begin{itemize}

\item  If we assume that $\omega(\sigma)=\sigma^{\alpha}$ for some $\alpha\in(0,1)$, and we have no informations about the derivative, then we can assume (as explained in Remark~\ref{rmk:limit-cases}) that $m(\lambda)\sim\lambda\,\omega(1/\lambda)=\lambda^{1-\alpha}$, which implies the classical result of~\cite{dgcs} according to which the problem is well-posed in Gevrey spaces of order $s\leq(1-\alpha)^{-1}$.

\item  If we assume that $\theta(t)=1/t^{\beta}$ for some $\beta>1$, and we have no information about the modulus of continuity, then we obtain (as explained in Remark~\ref{rmk:limit-cases}) that $m(\lambda)\sim\lambda^{(\beta-1)/\beta}$, which implies the classical result according to which the problem is well-posed in Gevrey spaces of order $s\leq\beta/(\beta-1)$ (see~\cite[Theorem~2]{2002-SNS-ColDSaKin}).

\item  If we ask both conditions, namely that $\omega(\sigma)=\sigma^{\alpha}$ and $\theta(t)=1/t^{\beta}$, then with some standard calculus we obtain that $m(\lambda)\sim\lambda^{(1-\alpha)(\beta-1)/\beta}$, and therefore the problem is well-posed in Gevrey spaces of order $s<\beta(\beta-1)^{-1}(1-\alpha)^{-1}$ (see~\cite[Theorem~3]{2002-SNS-ColDSaKin}). This is one more example of ``collaboration'' between the modulus of continuity and the control on the derivative in order to provide well-posedness results for less regular data.

\end{itemize}

The negative side is a derivative loss from Gevrey spaces to Gevrey distributions. For example, in the case where $\omega(\sigma)=\sigma^{\alpha}$ with $\alpha\in(0,1)$, and there are no informations on the derivative, there exist solutions whose initial data are in the Gevrey space of order $s$ for every $s>(1-\alpha)^{-1}$, and such that for positive times they do not belong to the space of Gevrey distributions of order $s$ for every $s>(1-\alpha)^{-1}$.

\end{em}
\end{rmk}

%\clearpage

\setcounter{equation}{0}
\section{Energy estimates from above}\label{sec:well-posed}

In this section we prove statement~(1) of Theorem~\ref{thm:main}, which implies in a standard way also statement~(1) of Theorem~\ref{thm:der-loss}.

To begin with, let us extend $c(t)$ to the whole half-line $t\geq 0$ by setting
\begin{equation}
\widehat{c}(t):=\begin{cases}
c(t)  & \text{if }t\leq T_{0}, \\
c(T_{0})  & \text{if }t\geq T_{0}.
\end{cases}
\nonumber
\end{equation}

For every $\ep>0$ let us set
\begin{equation}
c_{\ep}(t):=\frac{1}{\ep}\int_{t}^{t+\ep}\widehat{c}(s)\,ds
\qquad
\forall t\geq 0.
\nonumber
\end{equation}

Then it turns out that $c_{\ep}\in C^{1}([0,T_{0}])$ and satisfies the following estimates
\begin{equation}
\mu_{1}\leq c(t)\leq\mu_{2},
\qquad
|c_{\ep}(t)-c(t)|\leq\omega(\ep),
\qquad
|c_{\ep}'(t)|\leq\frac{\omega(\ep)}{\ep}
\label{est:cep}
\end{equation}
for every $t\in[0,T_{0}]$. Following~\cite{dgcs} we consider the usual Kovaleskyan energy
\begin{equation}
E_{\lambda}(t):=\ul'(t)^{2}+\lambda^{2}\ul(t)^{2},
\label{defn:E-kov}
\end{equation}
the usual hyperbolic energy
\begin{equation}
F_{\lambda}(t):=\ul'(t)^{2}+\lambda^{2}c(t)\ul(t)^{2},
\label{defn:E-hyp}
\end{equation}
and the approximated hyperbolic energy
\begin{equation}
F_{\ep,\lambda}(t):=\ul'(t)^{2}+\lambda^{2}c_{\ep}(t)\ul(t)^{2}.
\label{defn:E-hyp-mod}
\end{equation}

These energies are equivalent in the sense that
\begin{equation}
\min\{1,\mu_{1}\}E_{\lambda}(t)\leq F_{\lambda}(t)\leq\max\{1,\mu_{2}\}E_{\lambda}(t),
\label{equiv:kov-hyp}
\end{equation}
and
\begin{equation}
\min\{1,\mu_{1}\}E_{\lambda}(t)\leq F_{\ep,\lambda}(t)\leq\max\{1,\mu_{2}\}E_{\lambda}(t)
\label{equiv:kov-hyp-mod}
\end{equation}
for every admissible value of the parameters.  What we need in (\ref{th:main-above}) is an estimate of the Kovaleskyan energy (\ref{defn:E-kov}). To this end, for every $s\in(0,T_{0})$ we estimate the approximated hyperbolic energy in $[0,s]$, and the standard hyperbolic energy in $[s,T_{0}]$. 

The time-derivative of (\ref{defn:E-hyp-mod}) is
\begin{equation}
F_{\ep,\lambda}'(t)=c_{\ep}'(t)\lambda^{2}\ul(t)^{2}+\lambda^{2}(c_{\ep}(t)-c(t))\cdot 2\ul(t)\ul'(t),
\nonumber
\end{equation}
from which we deduce that
\begin{equation}
F_{\ep,\lambda}'(t)\leq\frac{|c_{\ep}'(t)|}{c_{\ep}(t)}F_{\ep,\lambda}(t)+
\lambda\frac{|c_{\ep}(t)-c(t)|}{c_{\ep}(t)^{1/2}}F_{\ep,\lambda}(t).
\nonumber
\end{equation}

Integrating this differential inequality, and keeping (\ref{est:cep}) into account, we deduce that
\begin{equation}
F_{\ep,\lambda}(t)\leq F_{\ep,\lambda}(0)
\exp\left\{\left(\frac{\omega(\ep)}{\mu_{1}\ep}+\lambda\,\frac{\omega(\ep)}{\sqrt{\mu_{1}}}\right)t\right\}
\qquad
\forall t\in[0,T_{0}].
\nonumber
\end{equation}

Setting $\ep:=1/\lambda$, and recalling (\ref{equiv:kov-hyp-mod}), this implies that
\begin{equation}
E_{\lambda}(t)\leq \sqrt{M_{1}}\,E_{\lambda}(0)
\exp\left\{M_{2}\,\lambda\,\omega\left(\frac{1}{\lambda}\right)s\right\}
\qquad\forall t\in[0,s],
\label{est:E-kov}
\end{equation}
where $M_{1}$ and $M_{2}$ are defined by (\ref{defn:M12}). The time-derivative of (\ref{defn:E-hyp}) is
\begin{equation}
F_{\lambda}'(t)=\lambda^{2}c'(t)|\ul(t)|^{2}\leq
\frac{|c'(t)|}{c(t)}F_{\lambda}(t)\leq
\frac{\theta(t)}{\mu_{1}}F_{\lambda}(t)
\qquad
\forall t\in(0,T_{0}].
\nonumber
\end{equation}

Integrating this differential inequality we deduce that
\begin{equation}
F_{\lambda}(t)\leq F_{\lambda}(s)\exp\left(\frac{1}{\mu_{1}}\int_{s}^{t}\theta(\tau)\,d\tau\right)
\leq F_{\lambda}(s)\exp\left(M_{2}\int_{s}^{T_{0}}\theta(\tau)\,d\tau\right)
\nonumber
\end{equation}
for every $t\in[s,T_{0}]$. Recalling the equivalence (\ref{equiv:kov-hyp}), and estimate (\ref{est:E-kov}) with $t=s$, we conclude that
\begin{eqnarray}
E_{\lambda}(t) & \leq &
\sqrt{M_{1}}\,E_{\lambda}(s)\exp\left(M_{2}\int_{s}^{T_{0}}\theta(\tau)\,d\tau\right)
\nonumber
\\[1ex]
& \leq & M_{1}\, E_{\lambda}(0)
\exp\left\{M_{2}\left(\lambda\,\omega\left(\frac{1}{\lambda}\right)s+\int_{s}^{T_{0}}\theta(\tau)\,d\tau\right)\right\}
\nonumber
\end{eqnarray}
for every $t\in[s,T_{0}]$. On the other hand, the same estimate holds true also for $t\in[0,s]$ because of (\ref{est:E-kov}). Optimizing with respect to $s$ we obtain exactly (\ref{th:main-above}).
\hfill$\Box$

%\clearpage

\setcounter{equation}{0}
\section{Counterexamples}\label{sec:counterexamples}

In this section we prove statement~(2) of Theorem~\ref{thm:main}, and we show how that statement leads to the counterexamples required for the optimality part in Theorem~\ref{thm:der-loss}.

\subsection{Asymptotic and universal activators}\label{sec:activators}

Let us begin by summarizing the theory developed in~\cite[section~4.1]{gg:CDSR-optimal}. In the sequel we consider solutions to the family of ordinary differential equations 
\begin{equation}
u_{\lambda}''(t)+\lambda^{2}\cl(t)u_{\lambda}(t)=0,
\label{eqn:uc-lambda}
\end{equation}
with initial data
\begin{equation}
u_{\lambda}(0)=0,
\qquad
u_{\lambda}'(0)=1.
\label{data:ode-lambda}
\end{equation}

We point out that in (\ref{eqn:uc-lambda}) the propagation speed depends on the parameter $\lambda$. When the propagation speed is fixed, we consider equation
\begin{equation}
v_{\lambda}''(t)+\lambda^{2}c(t)v_{\lambda}(t)=0,
\label{eqn:ode-lambda-n}
\end{equation}
with initial data
\begin{equation}
v_{\lambda}(0)=0,
\qquad
v_{\lambda}'(0)=1.
\label{data:ode-lambda-n}
\end{equation}

Let us recall our notion of activators (compare with~\cite{gg:DGCS-critical,gg:CDSR-optimal}).

\begin{defn}[Universal activators of a sequence]\label{defn:un-act}
\begin{em}

Let $T_{0}$ be a positive real number, let $\phi:(0,+\infty)\to(0,+\infty)$ be a function, and let $\{\lambda_{n}\}$ be a sequence of positive real numbers such that $\lambda_{n}\to +\infty$.

A \emph{universal activator} of the sequence $\{\lambda_{n}\}$ with rate $\phi$ is a coefficient $c\in L^{1}((0,T_{0}))$ such that the corresponding sequence $\{v_{\lambda_{n}}(t)\}$ of solutions to (\ref{eqn:ode-lambda-n})--(\ref{data:ode-lambda-n}) satisfies
\begin{equation}
\limsup_{n\to +\infty}\left(|v_{\lambda_{n}}'(t)|^{2}+\lambda_{n}^{2}|v_{\lambda_{n}}(t)|^{2}\right)
\exp\left(-\phi(\lambda_{n})\strut\right)\geq 1
\qquad
\forall t\in(0,T_{0}].
\label{th:defn-un-act}
\end{equation}

\end{em}
\end{defn}

\begin{defn}[Asymptotic activators]\label{defn:as-act}
\begin{em}

Let $T_{0}$ be a positive real number, and let $\phi:(0,+\infty)\to(0,+\infty)$ be a function. 

A \emph{family of asymptotic activators} with rate $\phi$ is a family of coefficients $\{c_{\lambda}(t)\}\subseteq L^{1}((0,T_{0}))$ with the property that, for every $\delta\in(0,T_{0})$, there exist two positive constants $M_{\delta}$ and $\lambda_{\delta}$ such that the corresponding family $\{u_{\lambda}(t)\}$ of solutions to (\ref{eqn:uc-lambda})--(\ref{data:ode-lambda}) satisfies
\begin{equation}
|u_{\lambda}'(t)|^{2}+\lambda^{2}|u_{\lambda}(t)|^{2}\geq
M_{\delta}\exp\left(2\phi(\lambda)\right)
\qquad
\forall t\in[\delta,T_{0}],\quad\forall\lambda\geq\lambda_{\delta}.
\label{eqn:asympt-activ}
\end{equation}

\end{em}
\end{defn}

The coefficient~2 in the exponential of (\ref{eqn:asympt-activ}) could be replaced by any number greater than~1. The following result shows that families of asymptotic activators are the basic tool in the construction of universal activators. This is the point where Baire category theorem discloses its power. For a proof, we refer to~\cite[Proposition~4.5]{gg:CDSR-optimal}.

\begin{prop}[From asymptotic to universal activators]\label{prop:as2un}

Let $\phi:(0,+\infty)\to(0,+\infty)$ be a function such that $\phi(\lambda)\to +\infty$ as $\lambda\to +\infty$. Let $T_{0}$ be a positive real number, and let $\PS\subseteq C^{0}([0,T_{0}])$ be a closed subset (with respect to uniform convergence).

Let us assume that there exists a dense subset $\mathcal{D}\subseteq\PS$ such that for every $c\in\mathcal{D}$ there exists a family of asymptotic activators $\{c_{\lambda}\}\subseteq\PS$ with rate $\phi$ such that $c_{\lambda}\to c$.

Then, for every unbounded sequence $\{\lambda_{n}\}$ of positive real numbers, the set of elements in $\PS$ that are universal activators of the sequence $\{\lambda_{n}\}$ with rate $\phi$ is residual in $\PS$ (and in particular nonempty).

\end{prop}

Finally, the following statement clarifies the crucial connection between universal activators and derivative loss. In order to show the strategy, we start by proving the result in the special case where $A$ admits an unbounded sequence of eigenvalues (see~\cite[Proposition~4.3]{gg:CDSR-optimal}).

\begin{prop}[Universal activators vs derivative loss -- Model case]\label{prop:ua2idl}

Let $H$ be a Hilbert space, and let $A$ be a nonnegative self-adjoint operator on $H$. Let us assume that there exists a sequence $\{e_{n}\}$ of orthonormal vectors in $H$, and an unbounded sequence of positive real numbers $\{\lambda_{n}\}$ such that $Ae_{n}=\lambda_{n}^{2}e_{n}$ for every positive integer $n$.

Let $T_{0}$ be a positive real number, let $\phi:(0,+\infty)\to(0,+\infty)$ be a function, and let $c\in L^{1}((0,T_{0}))$ be a universal activator of the sequence $\{\lambda_{n}\}$ with rate $\phi$. Let us assume also that
\begin{equation}
\phi(\lambda_{n})\geq n
\qquad
\forall n\geq 1.
\label{hp:series-ln}
\end{equation}

Then the asymptotic behavior of $\{\phi(\lambda_{n})\}$ determines the derivative loss of solutions to (\ref{defn:AWE})--(\ref{eqn:data}) according to the following scheme
\begin{align}
\lim_{n\to +\infty}\phi(\lambda_{n})\to +\infty &\quad\leadsto \quad
\text{(at least) arbitrarily small derivative loss},
\label{hp-th:asdl}
\\[1ex]
\liminf_{n\to +\infty}\frac{\phi(\lambda_{n})}{\log\lambda_{n}}>0 &\quad\leadsto \quad
\text{(at least) finite derivative loss},
\label{hp-th:fdl}
\\[1ex]
\lim_{n\to +\infty}\frac{\phi(\lambda_{n})}{\log\lambda_{n}}=+\infty &\quad\leadsto\quad 
\text{infinite derivative loss}.
\label{hp-th:idl}
\end{align}

\end{prop}

\begin{proof}

To begin with, we observe that (\ref{hp:series-ln}) implies that
\begin{equation}
\sum_{n=1}^{\infty}\exp\left(-\eta\phi(\lambda_{n})\right)<+\infty
\qquad
\forall\eta>0.
\label{hp:series-eta}
\end{equation}

For every positive integer $n$, let us set

\begin{equation}
a_{n}:=\exp\left(-\frac{\phi(\lambda_{n})}{4}\right),
\label{defn:an}
\end{equation}
and let us consider problem (\ref{defn:AWE})--(\ref{eqn:data}) with initial data
\begin{equation}
u_{0}:=0,
\qquad\qquad
u_{1}:=\sum_{n=1}^{\infty}a_{n}e_{n}.
\nonumber
\end{equation}

It is well-known that the unique solution is given by (a priori this series converges just in the sense of ultradistributions) 
\begin{equation}
u(t):=\sum_{n=1}^{\infty}a_{n}v_{\lambda_{n}}(t)e_{n}
\qquad
\forall t\in[0,T_{0}],
\nonumber
\end{equation}
where $\{v_{\lambda}(t)\}$ is the family of solutions to (\ref{eqn:ode-lambda-n})--(\ref{data:ode-lambda-n}). In particular, for every choice of the real numbers $\beta$ and $\gamma$ it turns out that
\begin{equation}
\sum_{n=1}^{\infty}\lambda_{n}^{4\beta}a_{n}^{2}=
\sum_{n=1}^{\infty}\exp\left(4\beta\log\lambda_{n}-\frac{\phi(\lambda_{n})}{2}\right),
\label{eqn:data-beta}
\end{equation}
and
\begin{multline}
\sum_{n=1}^{\infty}\lambda_{n}^{-4\gamma}a_{n}^{2}
\left(|v_{\lambda_{n}}'(t)|^{2}+\lambda_{n}^{2}|v_{\lambda_{n}}(t)|^{2}\right)
\\
=\sum_{n=1}^{\infty}\left(|v_{\lambda_{n}}'(t)|^{2}+\lambda_{n}^{2}|v_{\lambda_{n}}(t)|^{2}\right)
\exp(-\phi(\lambda_{n}))\cdot
\exp\left(\frac{\phi(\lambda_{n})}{2}-4\gamma\log\lambda_{n}\right).
\label{eqn:sol-gamma}
\end{multline}

Now we discuss the regularity of initial data by exploiting (\ref{eqn:data-beta}) and (\ref{hp:series-eta}), and the regularity of the corresponding solution $u(t)$ by exploiting (\ref{eqn:sol-gamma}) and condition (\ref{th:defn-un-act}) in the definition of universal activator. We distinguish three scenarios.

\begin{itemize}

\item  Under the assumption in (\ref{hp-th:asdl}) we observe that (\ref{eqn:data-beta}) converges when $\beta=0$, while (\ref{eqn:sol-gamma}) does not converge when $\gamma=0$. It follows that $(u_{0},u_{1})\in D(A^{1/2})\times H$, while $(u(t),u'(t))\not\in D(A^{1/2})\times H$ for every $t\in(0,T_{0}]$, which shows that in this case the solution $u(t)$ exhibits at least an arbitrarily small derivative loss.

\item  Under the assumption in (\ref{hp-th:fdl}), let $\delta\in(0,+\infty)$ denote the liminf of $\phi(\lambda_{n})/\log\lambda_{n}$. In this case we observe that (\ref{eqn:data-beta}) converges for every $\beta<\delta/8$, while (\ref{eqn:sol-gamma}) does not converge for every $\gamma<\delta/8$. As a consequence, the derivative loss of the solution $u(t)$ is at least $\delta/4$.

\item  Under the assumption in (\ref{hp-th:idl}) we observe that (\ref{eqn:data-beta}) converges for every $\beta\in\re$, and in particular $u_{1}\in D(A^{\infty})$, while (\ref{eqn:sol-gamma}) does not converge for every $\gamma\in\re$, which implies that the solution $u(t)$ has an infinite derivative loss.
\qedhere

\end{itemize}
\end{proof}

%\clearpage

In the following result we extend the construction of counterexamples to general unbounded self-adjoint operators.

\begin{prop}[Universal activators vs derivative loss -- General case]\label{prop:ua2dl-gen}

Let $H$ be a separable Hilbert space, and let $A$ be a nonnegative self-adjoint operator on $H$. Let $T_{0}$ be a positive real number, and let $\phi:(0,+\infty)\to(0,+\infty)$ be a function.

Let us assume that the operator $A$ is unbounded, and that $\phi(\lambda)\to +\infty$ as $\lambda\to +\infty$.

Then there exists a sequence of positive real numbers $\lambda_{n}\to +\infty$ with the following property. If $c\in L^{1}((0,T_{0}))$ is a universal activator of the sequence $\{\lambda_{n}\}$ with rate $\phi$, then the asymptotic behavior of $\{\phi(\lambda_{n})\}$ determines the derivative loss of solutions to (\ref{defn:AWE})--(\ref{eqn:data}) according to the scheme (\ref{hp-th:asdl}) through (\ref{hp-th:idl}).

\end{prop}

\begin{proof}

We imitate the proof of Proposition~\ref{prop:ua2idl} by exploiting the general form of the spectral theorem and a reinforced version of universal activators.

\paragraph{\textmd{\textit{Definition of the sequence $\{\lambda_{n}\}$}}}

According to the spectral theorem for self-adjoint operators (see for example~\cite[Theorem~VIII.4]{reed-simon}) there exists a finite measure space $(M,\mu)$, an isometric bijective map $\ft:H\to L^{2}(M,\mu)$, and a measurable function $\lambda:M\to[0,+\infty)$ such that the operator $A$ on $H$ acts as the multiplication operator by $\lambda^{2}$ in $L^{2}(M,\mu)$. 

More precisely, to every vector $w\in H$ it is associated the ``generalized Fourier transform'' $\what(\xi):=\ft(v)\in L^{2}(M,\mu)$ in such a way that
\begin{itemize}

\item  $w\in D(A^{\beta})$ if and only if $(1+\lambda(\xi)^{2})^{\beta}\,\what(\xi)\in L^{2}(M,\mu)$,

\item  if $w\in D(A)$ then $[\ft(Aw)](\xi)=\lambda(\xi)^{2}\,\what(\xi)$ for $\mu$-almost every $\xi\in M$.

\end{itemize}

Since $A$ is unbounded, the function $\lambda(\xi)$ is essentially unbounded, and therefore there exists a sequence of positive real numbers $\lambda_{n}\to +\infty$ such that
\begin{equation}
\mu(\{\xi\in M:\lambda(\xi)\in[\lambda_{n}-r,\lambda_{n}+r]\})>0
\qquad
\forall n\geq 1
\quad
\forall r>0.
\nonumber
\end{equation}

Up to passing to a subsequence, we can also assume that the sequence $\{\lambda_{n}\}$ is strictly increasing and (\ref{hp:series-ln}) holds true.

\paragraph{\textmd{\textit{A ``stronger'' property of universal activators}}}

Let $\{v_{\lambda}(t)\}$ denote the family of solutions to (\ref{eqn:ode-lambda-n})--(\ref{data:ode-lambda-n}). We claim that there exists a sequence $\{r_{n}\}$ of positive real numbers such that the intervals $[\lambda_{n}-r_{n},\lambda_{n}+r_{n}]$ are pairwise disjoint and, if we set
\begin{equation}
I_{n}(t):=\inf\left\{v_{\lambda}'(t)^{2}+\lambda^{2}v_{\lambda}(t)^{2}\strut:
\lambda\in[\lambda_{n}-r_{n},\lambda_{n}+r_{n}]\right\},
\label{defn:In(t)}
\end{equation}
then it turns out that
\begin{equation}
\limsup_{n\to +\infty}I_{n}(t)\exp\left(-\phi(\lambda_{n})\right)
\geq 1
\qquad
\forall t\in(0,T_{0}].
\label{defn:ua-strong}
\end{equation}

To this end, it is enough to observe that the map
\begin{equation}
(0,+\infty)\ni\lambda\mapsto v_{\lambda}(t)\in C^{1}([0,T_{0}])
\nonumber
\end{equation}
is continuous, and then choose $r_{n}$ in such a way that  
\begin{equation}
v_{\lambda}'(t)^{2}+\lambda^{2}v_{\lambda}(t)^{2}\geq
v_{\lambda_{n}}'(t)^{2}+\lambda_{n}^{2}v_{\lambda_{n}}(t)^{2}-1
\nonumber
\end{equation}
for every $t\in[0,T_{0}]$ and every $\lambda\in[\lambda_{n}-r_{n},\lambda_{n}+r_{n}]$. At this point, (\ref{defn:ua-strong}) follows from (\ref{th:defn-un-act}) because $\phi(\lambda_{n})\to +\infty$. Up to reducing $r_{n}$ if necessary, we can also assume that
\begin{equation}
\frac{\lambda_{n}}{2}\leq\lambda_{n}-r_{n}<\lambda_{n}+r_{n}\leq 2\lambda_{n}
\qquad
\forall n\geq 1.
\label{est:rn-ln}
\end{equation}

\paragraph{\textmd{\textit{Construction of counterexamples}}}

Let $\{\lambda_{n}\}$ be any sequence as in the first paragraph, and let $c\in L^{1}((0,T_{0}))$ be any universal activator of the sequence $\{\lambda_{n}\}$ with rate $\phi$. We need to show that problem (\ref{defn:AWE})--(\ref{eqn:data}) exhibits the prescribed derivative loss. To this end, for every positive integer $n$ we define $r_{n}$ as in the previous paragraph, we consider the set
\begin{equation}
M_{n}:=\left\{\xi\in M:\lambda(\xi)\in[\lambda_{n}-r_{n},\lambda_{n}+r_{n}]\right\},
\nonumber
\end{equation}
which has positive measure, and we call $\what_{n}(\xi)$ the characteristic function of $M_{n}$ (namely $\what_{n}(\xi)=1$ if $\xi\in M_{n}$, and $\what_{n}(\xi)=0$ otherwise). Then we define $a_{n}$ as in (\ref{defn:an}), we set
\begin{equation}
\widehat{u}_{1}(\xi):=\sum_{n=1}^{\infty}\frac{a_{n}}{\mu(M_{n})^{1/2}}\,\what_{n}(\xi),
\nonumber
\end{equation}
and we consider problem (\ref{defn:AWE})--(\ref{eqn:data}) with initial data $u_{0}=0$ and $u_{1}:=\ft^{-1}(\widehat{u}_{1}(\xi))$. It is well-known that the ``Fourier transform'' of the solution is
\begin{equation}
\widehat{u}(t,\xi):=[\ft(u(t))](\xi)=\widehat{u}_{1}(\xi)\cdot v_{\lambda(\xi)}(t),
\nonumber
\end{equation}
where $\{v_{\lambda}(t)\}$ is again the family of solutions to (\ref{eqn:ode-lambda-n})--(\ref{data:ode-lambda-n}). 

Let us examine the regularity of $u_{1}$ and of the pair $(u(t),u'(t))$. As for the regularity of $u_{1}$, for every real number $\beta\geq 0$ it turns out that
\begin{equation}
u_{1}\in D(A^{\beta})
\quad\Longleftrightarrow\quad
\int_{M}\lambda(\xi)^{4\beta}\,\widehat{u}_{1}(\xi)^{2}\,d\mu(\xi)<+\infty.
\nonumber
\end{equation}

On the other hand, since the sets $M_{n}$ are pairwise disjoint, we obtain that
\begin{multline}
\int_{M}\lambda(\xi)^{4\beta}\,\widehat{u}_{1}(\xi)^{2}\,d\mu(\xi)=
\sum_{n=1}^{\infty}\int_{M_{n}}\lambda(\xi)^{4\beta}\,\widehat{u}_{1}(\xi)^{2}\,d\mu(\xi)%=
\\
=\sum_{n=1}^{\infty}\frac{a_{n}^{2}}{\mu(M_{n})}\int_{M_{n}}\lambda(\xi)^{4\beta}\,d\mu(\xi)
\leq\sum_{n=1}^{\infty}a_{n}^{2}(\lambda_{n}+r_{n})^{4\beta}
\leq 2^{4\beta}\sum_{n=1}^{\infty}a_{n}^{2}\lambda_{n}^{4\beta},
\nonumber
\end{multline}
where in the last step we exploited the estimate from above in (\ref{est:rn-ln}), and analogously
\begin{equation}
\int_{M}\lambda(\xi)^{4\beta}\,\widehat{u}_{1}(\xi)^{2}\,d\mu(\xi)\geq
\frac{1}{2^{4\beta}}\sum_{n=1}^{\infty}a_{n}^{2}\lambda_{n}^{4\beta},
\nonumber
\end{equation}
so that in conclusion
\begin{equation}
u_{1}\in D(A^{\beta})
\quad\Longleftrightarrow\quad
\sum_{n=1}^{\infty}a_{n}^{2}\lambda_{n}^{4\beta}<+\infty.
\nonumber
\end{equation}

As for the regularity of $u(t)$, we observe that for every real number $\gamma\geq 0$ it turns out that $(u(t),u'(t))\in D(A^{-\gamma+1/2})\times D(A^{-\gamma})$ if and only if
\begin{equation}
\int_{M}\lambda(\xi)^{-4\gamma}\,\widehat{u}_{1}(\xi)^{2}
\left(v_{\lambda(\xi)}'(t)^{2}+\lambda(\xi)^{2}\,v_{\lambda(\xi)}(t)^{2}\right)
\,d\mu(\xi)<+\infty.
\nonumber
\end{equation}

Recalling (\ref{defn:In(t)}) the last integral can be rewritten and estimated as follows
\begin{eqnarray*}
\lefteqn{\hspace{-5em}
\sum_{n=1}^{\infty}\frac{a_{n}^{2}}{\mu(M_{n})}\int_{M_{n}}
\lambda(\xi)^{-4\gamma}\left(v_{\lambda(\xi)}'(t)^{2}+\lambda(\xi)^{2}\,v_{\lambda(\xi)}(t)^{2}\right)
\,d\mu(\xi)}
\\
\qquad\qquad& \geq & 
\sum_{n=1}^{\infty}(\lambda_{n}+r_{n})^{-4\gamma}a_{n}^{2}I_{n}(t)
\\
& \geq & \frac{1}{2^{4\gamma}}\sum_{n=1}^{\infty}\lambda_{n}^{-4\gamma}a_{n}^{2}I_{n}(t)
\\
& = & 
\frac{1}{2^{4\gamma}}\sum_{n=1}^{\infty}I_{n}(t)
\exp(-\phi(\lambda_{n}))\cdot
\exp\left(\frac{\phi(\lambda_{n})}{2}-4\gamma\log\lambda_{n}\right),
\end{eqnarray*}
where in the last inequality we exploited the estimate from above in (\ref{est:rn-ln}).

Thanks to (\ref{defn:ua-strong}), at this point all conclusions follow as in Proposition~\ref{prop:ua2idl}.
\end{proof}

%\clearpage

\subsection{Building block and dense subset}

From the general theory, we know that we need to show that asymptotic activators can approximate all coefficients in a dense subset. In this subsection we identify this dense subset, and then we describe the starting point of the construction of the approximating family of asymptotic activators.

\begin{defn}\label{defn:D}
\begin{em}

Let $T_{0}$, $\mu_{1}$, $\mu_{2}$, $\omega$, $\theta$ be as in Definition~\ref{defn:PS}. We call $\mathcal{D}$ the set of all functions $c_{*}\in\PS(T_{0},\mu_{1},\mu_{2},\omega,\theta)$ for which there exist real numbers $T_{1}$, $\gamma$, and $\eta$ (that might depend on $c_{*}$) with
\begin{equation}
T_{1}\in(0,T_{0}),
\qquad
\gamma^{2}\in(\mu_{1},\mu_{2}),
\qquad
\eta\in(0,1)
\nonumber
\end{equation}
such that
\begin{equation}
c_{*}(t)=\gamma^{2}
\qquad
\forall t\in[0,T_{1}]
\nonumber
\end{equation}
and
\begin{equation}
|c_{*}(t)-c_{*}(s)|\leq(1-\eta)\omega(|t-s|)
\qquad
\forall(t,s)\in[0,T_{0}]^{2}.
\label{defn:c*-omega}
\end{equation}

When we want to emphasize the parameters we write $c\in\mathcal{D}(T_{0},\mu_{1},\mu_{2},\omega,\theta;T_{1},\gamma,\eta)$.

\end{em}
\end{defn}

In word, the elements of $\mathcal{D}$ are constant in a right neighborhood of the origin, and they do not saturate neither the strict hyperbolicity condition in this neighborhood, nor the inequality in the definition of $\omega$-continuity. As one can easily guess, the result is that these special coefficients are dense in the classes introduced in Definition~\ref{defn:PS}.

\begin{prop}[Density]
 
The set $\mathcal{D}$ is dense in $\PS(T_{0},\mu_{1},\mu_{2},\omega,\theta)$ for every admissible choice of the parameters.

\end{prop}

\begin{proof}

Let $c(t)$ be any element of $\PS(T_{0},\mu_{1},\mu_{2},\omega,\theta)$. For every $\ep\in(0,1)$, with $\ep<T_{0}$, let us set
\begin{equation}
c_{\ep}(t):=\begin{cases}
(1-\ep)c(\ep)+\ep\,\dfrac{\mu_{2}+\mu_{1}}{2} & \text{if }t\in[0,\ep], 
\\[2ex]
(1-\ep)c(t)+\ep\,\dfrac{\mu_{2}+\mu_{1}}{2} & \text{if }t\in[\ep,T_{0}].
\end{cases}
\nonumber
\end{equation} 

Then it turns out that $c_{\ep}\in\mathcal{D}(T_{0},\mu_{1},\mu_{2},\omega,\theta;T_{1},\gamma,\eta)$ with \begin{equation}
T_{1}:=\ep,
\qquad\quad
\gamma:=\sqrt{c_{\ep}(0)}=\sqrt{c_{\ep}(\ep)},
\qquad\quad
\eta:=\ep,
\nonumber
\end{equation}
and that $c_{\ep}(t)\to c(t)$ uniformly in $[0,T_{0}]$.
\end{proof}

%\clearpage

The following lemma is essentially taken form~\cite{dgcs}. We state and prove it because we need the exact values of the constants.

\begin{lemma}[Basic block]\label{lemma:altalena}
	
Let $\ep$, $\gamma$, $\lambda$ be positive real numbers such that
\begin{equation}
\ep\leq 8\gamma^{3}\lambda.
\label{hp:ep-gl}
\end{equation}

Let $(a,b)\subseteq\re$ be an interval whose endpoints satisfy
\begin{equation}
\frac{a\gamma\lambda}{2\pi}\in\n
\qquad\text{and}\qquad
\frac{b\gamma\lambda}{2\pi}\in\n.
\label{hp:ab-int}
\end{equation}

For every $t\in[a,b]$ let us set
\begin{equation}
\varphi_{\ep,\gamma,\lambda}(t)
:=\frac{\ep}{4\gamma\lambda}\sin(2\gamma\lambda t)+
\frac{\ep^{2}}{64\gamma^{4}\lambda^{2}}\sin^{4}(\gamma\lambda t)
\label{defn:varphi}
\end{equation}
and
\begin{equation}
w_{\ep,\gamma,\lambda}(t):=\frac{1}{\gamma\lambda}\sin(\gamma\lambda t)
\exp\left(\frac{\ep}{16\gamma^{2}}(t-a)-\frac{\ep}{32\gamma^{3}\lambda}\sin(2\gamma\lambda t)\right).
\label{defn:w}
\end{equation}

Then the following statements hold true.
\begin{enumerate}
\renewcommand{\labelenumi}{(\arabic{enumi})}

\item  For every $t\in[a,b]$ it turns out that
\begin{equation}
|\varphi_{\ep,\gamma,\lambda}(t)|\leq\frac{\ep}{2\gamma\lambda}
\qquad\quad\text{and}\quad\qquad
\left|\varphi_{\ep,\gamma,\lambda}'(t)\right|\leq\ep.
\label{th:est-varphi}
\end{equation}

\item  For every modulus of continuity $\omega$ it turns out that
\begin{equation}
\left|\varphi_{\ep,\gamma,\lambda}(t)-\varphi_{\ep,\gamma,\lambda}(s)\right|\leq
\ep\cdot\max\left\{1,\frac{\pi}{\gamma}\right\}\cdot
\left[\lambda\,\omega\left(\frac{1}{\lambda}\right)\right]^{-1}\cdot
\omega(|t-s|)
\label{th:varphi-omega}
\end{equation}
for every $s$ and $t$ in $[a,b]$.

\item The function $w_{\ep,\gamma,\lambda}(t)$ satisfies the differential equation 
\begin{equation}
w_{\ep,\gamma,\lambda}''(t)+
\lambda^{2}\left(\gamma^{2}-\varphi_{\ep,\gamma,\lambda}(t)\right)\cdot w_{\ep,\gamma,\lambda}(t)=0
\nonumber
\end{equation}
for every $t\in[a,b]$, with ``initial'' data $w_{\ep,\gamma,\lambda}(a)=0$, $w_{\ep,\gamma,\lambda}'(a)=1$, and ``final'' data 
\begin{equation}
w_{\ep,\gamma,\lambda}(b)=0,
\qquad
w_{\ep,\gamma,\lambda}'(b)=\exp\left(\frac{\ep(b-a)}{16\gamma^{2}}\right).
\nonumber
\end{equation}

\end{enumerate}

\end{lemma}

\begin{proof}

Let us start with statement~(1). From definition (\ref{defn:varphi}) it follows that
\begin{equation}
|\varphi_{\ep,\gamma,\lambda}(t)|\leq
\frac{\ep}{4\gamma\lambda}\left(1+\frac{\ep}{16\gamma^{3}\lambda}\right)
\qquad\text{and}\qquad
|\varphi_{\ep,\gamma,\lambda}'(t)|\leq
\frac{\ep}{2}\left(1+\frac{\ep}{8\gamma^{3}\lambda}\right),
\nonumber
\end{equation}
and these estimates imply (\ref{th:est-varphi}) because of assumption (\ref{hp:ep-gl}).

As for statement~(3), it is just a (lengthy) computation.

It remains to prove statement~(2). Let $t$ and $s$ be in $[a,b]$. Since the function $\varphi_{\ep,\gamma,\lambda}$ is periodic with period $\pi/(\gamma\lambda)$, there exists $t_{1}$ and $s_{1}$ in $[a,b]$ such that
\begin{equation}
\varphi_{\ep,\gamma,\lambda}(t)=\varphi_{\ep,\gamma,\lambda}(t_{1}),
\qquad\qquad
\varphi_{\ep,\gamma,\lambda}(s)=\varphi_{\ep,\gamma,\lambda}(s_{1}),
\nonumber
\end{equation}
and
\begin{equation}
|t_{1}-s_{1}|\leq\frac{\pi}{\gamma\lambda},
\qquad\qquad
|t_{1}-s_{1}|\leq|t-s|.
\label{est:t1s1}
\end{equation}

From the second estimate in (\ref{th:est-varphi}) we know that $\varphi_{\ep,\gamma,\lambda}$ is Lipschitz continuous with Lipschitz constant less than or equal to $\ep$, and in particular
\begin{eqnarray}
\left|\varphi_{\ep,\gamma,\lambda}(t)-\varphi_{\ep,\gamma,\lambda}(s)\right| & = &
\left|\varphi_{\ep,\gamma,\lambda}(t_{1})-\varphi_{\ep,\gamma,\lambda}(s_{1})\right|
\nonumber
\\[0.5ex]
& \leq & \ep\cdot|t_{1}-s_{1}|
\nonumber
\\[0.5ex]
& = & \ep\cdot\frac{|t_{1}-s_{1}|}{\omega(|t_{1}-s_{1}|)}\cdot\omega(|t_{1}-s_{1}|)
\nonumber
\\[0.5ex]
& \leq &
\ep\cdot\frac{\pi}{\gamma\lambda}\left[\omega\left(\frac{\pi}{\gamma\lambda}\right)\right]^{-1}
\cdot\omega(|t-s|),
\label{est:varphi-omega}
\end{eqnarray}
where in the last step we exploited that the functions $\omega(\sigma)$ and $\sigma/\omega(\sigma)$ are nondecreasing, and inequalities (\ref{est:t1s1}). Now we distinguish two cases.
\begin{itemize}

\item  If $\pi/\gamma\geq 1$, then $\omega(\pi/(\gamma\lambda))\geq\omega(1/\lambda)$ because of the monotonicity of $\omega(\sigma)$. It follows that
\begin{equation}
\frac{\pi}{\gamma\lambda}\left[\omega\left(\frac{\pi}{\gamma\lambda}\right)\right]^{-1}\leq
\frac{\pi}{\gamma}\cdot\frac{1}{\lambda}\left[\omega\left(\frac{1}{\lambda}\right)\right]^{-1},
\nonumber
\end{equation}
and therefore in this case (\ref{th:varphi-omega}) follows from (\ref{est:varphi-omega}).

\item  If $\pi/\gamma\leq 1$, then $\pi/(\gamma\lambda)\leq 1/\lambda$, and exploiting again the monotonicity of $\sigma/\omega(\sigma)$ we obtain that
\begin{equation}
\frac{\pi}{\gamma\lambda}\left[\omega\left(\frac{\pi}{\gamma\lambda}\right)\right]^{-1}\leq
\frac{1}{\lambda}\left[\omega\left(\frac{1}{\lambda}\right)\right]^{-1},
\nonumber
\end{equation}
and therefore also in this case (\ref{th:varphi-omega}) follows from (\ref{est:varphi-omega}).
\qedhere

\end{itemize}
\end{proof}

%\clearpage

\subsection{Proof of Theorem~\ref{thm:main}, statement~(2)}

It remains to prove that, for every coefficient $c_{*}(t)$ in the dense subset $\mathcal{D}$ described in Definition~\ref{defn:D}, there exists a family of asymptotic activators that converge uniformly to $c_{*}(t)$. This result is proved in Proposition~\ref{prop:final}, and the proof relies on two preliminary general constructions, that we introduce in the following two propositions.

\begin{prop}[$\omega$-construction]\label{prop:omega}

Let us assume that $c_{*}\in\mathcal{D}(T_{0},\mu_{1},\mu_{2},\omega,\theta;T_{1},\gamma,\eta)$ for some admissible values of the parameters, and let us set
\begin{equation}
\nu_{1}:=\min\left\{1,\frac{\sqrt{\mu_{1}}}{\pi}\right\}.
\label{defn:delta}
\end{equation}

Let $\lambda$ be a positive real number such that
\begin{equation}
\nu_{1}\,\omega\left(\frac{1}{\lambda}\right)\leq 8\gamma^{3},
\qquad\qquad
\frac{\nu_{1}}{2\gamma}\,\omega\left(\frac{1}{\lambda}\right)\leq
\min\left\{\gamma^{2}-\mu_{1},\mu_{2}-\gamma^{2}\right\}.
\label{hp:lo-lambda}
\end{equation}

Let $(a,b)\subseteq(0,T_{1})$ be an interval whose endpoints satisfy (\ref{hp:ab-int}) and
\begin{equation}
\omega(b)\leq\eta\,\omega(T_{1}-b),
\qquad\qquad
\nu_{1}\lambda\,\omega\left(\frac{1}{\lambda}\right)\leq\theta(b).
\label{hp:lo-theta}
\end{equation}

Finally, let consider the function $\varphi_{\ep,\lambda,\gamma}(t)$ defined in (\ref{defn:w}) with 
\begin{equation}
\ep:=\nu_{1}\lambda\,\omega\left(\frac{1}{\lambda}\right),
\nonumber
\end{equation}
and let us define
\begin{equation}
c_{\lambda}(t):=\begin{cases}
c_{*}(t)  & \text{if }t\in[0,a]\cup[b,T_{0}], 
\\[0.5ex]
\gamma^{2}-\varphi_{\ep,\gamma,\lambda}(t)\quad  & 
\text{if }t\in[a,b].
\end{cases}
\nonumber
\end{equation}

Then the following statements hold true.
\begin{enumerate}
\renewcommand{\labelenumi}{(\arabic{enumi})}

\item  The function $c_{\lambda}$ belongs to $\PS(T_{0},\mu_{1},\mu_{2},\omega,\theta)$ and satisfies
\begin{equation}
|c_{\lambda}(t)-c_{*}(t)|\leq\frac{\nu_{1}}{2\gamma}\,\omega\left(\frac{1}{\lambda}\right)
\qquad
\forall t\in[0,T_{0}].
\label{th:lo-cl-c*}
\end{equation}

\item  The solution $\ul(t)$ to problem (\ref{eqn:uc-lambda})--(\ref{data:ode-lambda}) satisfies
\begin{multline}
\ul'(t)^{2}+\lambda^{2}\ul(t)^{2}
\\
\geq
\min\left\{1,\frac{1}{\mu_{2}}\right\}
\exp\left(-\frac{1}{\mu_{1}}\int_{T_{1}}^{T_{0}}\theta(t)\,dt\right)
\exp\left(\frac{\nu_{1}}{8\mu_{2}}\lambda\,\omega\left(\frac{1}{\lambda}\right)(b-a)\right)
\label{th:lo-ul}
\end{multline}
for every $t\in[b,T_{0}]$.

\end{enumerate}

\end{prop}

\begin{proof}

To begin with, we observe that the first inequality in (\ref{hp:lo-lambda}) is equivalent to $\ep\leq 8\gamma^{3}\lambda$, and therefore $\ep$, $\gamma$, $\lambda$ satisfy the assumptions of Lemma~\ref{lemma:altalena}.

\paragraph{\textmd{\textit{Statement~(1)}}}

Let us start by proving (\ref{th:lo-cl-c*}). To this end, we can assume that $t\in[a,b]$, because otherwise  $\cl(t)$ and $c_{*}(t)$ coincide. When $t\in[a,b]$, from the first estimate in (\ref{th:est-varphi}) we obtain that
\begin{equation}
|\cl(t)-c_{*}(t)|=
|\varphi_{\ep,\gamma,\lambda}(t)|\leq
\frac{\ep}{2\gamma\lambda}=
\frac{\nu_{1}}{2\gamma}\,\omega\left(\frac{1}{\lambda}\right),
\nonumber
\end{equation}
which proves (\ref{th:lo-cl-c*}).

Now let us prove that $c_{\lambda}\in\PS(T_{0},\mu_{1},\mu_{2},\omega,\theta)$. As for the strict hyperbolicity condition, we can limit ourselves to the interval $[a,b]$, where it follows from (\ref{th:lo-cl-c*}) because of the second condition in (\ref{hp:lo-lambda}).

As for the estimate on the derivative, again we can limit ourselves to the interval $[a,b]$. In this case from the second estimate in (\ref{th:est-varphi}) and assumption (\ref{hp:lo-theta}) we obtain that
\begin{equation}
|\cl'(t)|=
|\varphi_{\ep,\gamma,\lambda}'(t)|\leq
\ep=
\nu_{1}\lambda\,\omega\left(\frac{1}{\lambda}\right)\leq
\theta(b)\leq
\theta(t),
\nonumber
\end{equation}
where the last inequality follows from the monotonicity of $\theta(t)$.

Finally, let us check the $\omega$-continuity of $\cl(t)$. To this end, we consider $t$ and $s$ in $[0,T_{0}]$, we assume without loss of generality that $s<t$, and we distinguish some cases according to the position of $t$ and $s$.
\begin{itemize}

\item  If $a\leq s<t\leq b$, then we exploit (\ref{th:varphi-omega}), and from our definition (\ref{defn:delta}) of $\nu_{1}$ we obtain that
\begin{equation}
|\cl(t)-\cl(s)|=
|\varphi_{\ep,\gamma,\lambda}(t)-\varphi_{\ep,\gamma,\lambda}(s)|\leq
\nu_{1}\cdot\max\left\{1,\frac{\pi}{\gamma}\right\}\omega(t-s)\leq
\omega(t-s).
\label{th:omega-ab}
\end{equation}

\item  If $b\leq s<t\leq T_{0}$, then the $\omega$-continuity of $\cl$ follows from the $\omega$-continuity of $c_{*}$.

\item  If $s\in[a,b]$ and $t\in[T_{1},T_{0}]$, then
\begin{eqnarray*}
|\cl(t)-\cl(s)| & \leq &
|\cl(t)-\cl(b)|+|\cl(b)-\cl(s)|
\\[0.5ex]
& = & |c_{*}(t)-c_{*}(b)|+|\varphi_{\ep,\gamma,\lambda}(b)-\varphi_{\ep,\gamma,\lambda}(s)|
\\[0.5ex]
& \leq & (1-\eta)\omega(t-b)+\omega(b-s)
\\[0.5ex]
& \leq & (1-\eta)\omega(t-b)+\omega(b),
\end{eqnarray*} 
where we exploited that $c_{*}$ satisfies (\ref{defn:c*-omega}) in $[b,T_{0}]$, and the fact that $\cl$ satisfies (\ref{th:omega-ab}) in $[s,b]$. At this point we exploit the first inequality in (\ref{hp:lo-theta}) and we conclude that
\begin{equation}
|\cl(t)-\cl(s)|\leq 
\omega(t-b)+\eta[\omega(T_{1}-b)-\omega(t-b)]\leq
\omega(t-b)\leq
\omega(t-s).
\nonumber
\end{equation}

\item  The cases where at least one variable lies in $[0,a]\cup[b,T_{1}]$ are either trivial or can be easily reduced to the previous ones. 

\end{itemize} 

\paragraph{\textmd{\textit{Statement~(2)}}}

Let us examine now the solution to problem (\ref{eqn:uc-lambda})--(\ref{data:ode-lambda}). In the interval $[0,a]$ the solution is given by the explicit formula
\begin{equation}
\ul(t)=\frac{1}{\gamma\lambda}\sin(\gamma\lambda t),
\nonumber
\end{equation}
and hence, since $\gamma\lambda a$ is an integer multiple of $2\pi$, it follows that $\ul(a)=0$ and $\ul'(a)=1$.

In the interval $[a,b]$ the solution is given by the explicit formula $\ul(t)=w_{\ep,\gamma,\lambda}(t)$, where  $w_{\ep,\gamma,\lambda}$ is defined by (\ref{defn:w}). Since $\gamma\lambda b$ is an integer multiple of $2\pi$, from the explicit formula we obtain that
\begin{equation}
\ul'(b)^{2}+\lambda^{2}\gamma^{2}\ul(b)^{2}=
\exp\left(\frac{\nu_{1}}{8\gamma^{2}}\lambda\,\omega\left(\frac{1}{\lambda}\right)(b-a)\right)\geq
\exp\left(\frac{\nu_{1}}{8\mu_{2}}\lambda\,\omega\left(\frac{1}{\lambda}\right)(b-a)\right).
\label{eqn:Fl(b)}
\end{equation}

Finally, in the interval $[b,T_{0}]$ we consider the classical hyperbolic energy
\begin{equation}
F_{\lambda}(t):=\ul'(t)^{2}+\lambda^{2}\cl(t)\ul(t)^{2}.
\nonumber
\end{equation}

In the usual way we obtain that
\begin{equation}
\ul'(t)^{2}+\lambda^{2}\ul(t)^{2}\geq
\min\left\{1,\frac{1}{\mu_{2}}\right\}F_{\lambda}(t),
\label{Fl-equiv}
\end{equation}
and
\begin{equation}
F_{\lambda}'(t)=\lambda^{2}\cl'(t)\ul(t)^{2}\geq
-\frac{|\cl'(t)|}{\cl(t)}\,F_{\lambda}(t)=
-\frac{|c_{*}'(t)|}{c_{*}(t)}\,F_{\lambda}(t)\geq
-\frac{|c_{*}'(t)|}{\mu_{1}}\,F_{\lambda}(t).
\nonumber
\end{equation}

Since $c_{*}'(t)=0$ in $[b,T_{1}]$, and $|c_{*}'(t)|\leq\theta(t)$ in $[T_{1},T_{0}]$, integrating this differential inequality we obtain that
\begin{equation}
F_{\lambda}(t)\geq 
F_{\lambda}(b)\exp\left(-\frac{1}{\mu_{1}}\int_{T_{1}}^{T_{0}}\theta(\tau)\,d\tau\right)
\qquad
\forall t\in[b,T_{0}].
\label{est:Fl>exp}
\end{equation}

Since $F_{\lambda}(b)$ is given by (\ref{eqn:Fl(b)}), at this point (\ref{th:lo-ul}) follows from (\ref{Fl-equiv}) and (\ref{est:Fl>exp}).
\end{proof}

%\clearpage

\begin{prop}[$\theta$-construction]\label{prop:theta}

Let us assume that $c_{*}\in\mathcal{D}(T_{0},\mu_{1},\mu_{2},\omega,\theta;T_{1},\gamma,\eta)$ for some admissible values of the parameters, and let us set
\begin{equation}
\nu_{2}:=\frac{1}{2}\min\left\{1,\frac{\sqrt{\mu_{1}}}{\pi}\right\}.
\label{defn:lt-delta}
\end{equation}

Let $\lambda$ be a positive real number such that
\begin{equation}
\nu_{2}\,\omega\left(\frac{1}{\lambda}\right)\leq 8\gamma^{3},
\qquad\qquad
\frac{\nu_{2}}{2\gamma}\,\omega\left(\frac{1}{\lambda}\right)\leq
\min\left\{\gamma^{2}-\mu_{1},\mu_{2}-\gamma^{2}\right\}.
\label{hp:lt-lambda}
\end{equation}

Let $(a,b)\subseteq(0,T_{1})$ be an interval whose endpoints satisfy (\ref{hp:ab-int}) and
\begin{equation}
\omega(b)\leq\eta\,\omega(T_{1}-b),
\qquad\qquad
\lambda\,\omega\left(\frac{1}{\lambda}\right)\geq\theta(a).
\label{hp:lt-theta}
\end{equation}

Let us set $t_{0}:=a$ and $k:=\gamma\lambda(b-a)/(2\pi)$, and then let us define
\begin{equation}
t_{i}:=a+\frac{2\pi}{\gamma\lambda}\cdot i,
\qquad\text{and}\qquad
\ep_{i}:=\nu_{2}\,\theta(t_{i})
\qquad\quad
\forall i\in\{1,\ldots,k\}.
\nonumber
\end{equation}

Finally, let us define
\begin{equation}
c_{\lambda}(t):=\begin{cases}
c_{*}(t)  & \text{if }t\in[0,a]\cup[b,T_{0}], 
\\[0.5ex]
\gamma^{2}-\varphi_{\ep,\gamma,\lambda}(t)\quad  & 
\text{if $t\in[t_{i-1},t_{i}]$ for some $i\in\{1,\ldots,k\}$}.
\end{cases}
\nonumber
\end{equation}

Then the following statements hold true.
\begin{enumerate}
\renewcommand{\labelenumi}{(\arabic{enumi})}

\item  The function $c_{\lambda}$ belongs to $\PS(T_{0},\mu_{1},\mu_{2},\omega,\theta)$ and satisfies
\begin{equation}
|c_{\lambda}(t)-c_{*}(t)|\leq\frac{\nu_{2}}{2\gamma}\,\omega\left(\frac{1}{\lambda}\right)
\qquad
\forall t\in[0,T_{0}].
\label{th:lt-cl-c*}
\end{equation}

\item  The solution to problem (\ref{eqn:uc-lambda})--(\ref{data:ode-lambda}) satisfies
\begin{multline}
\ul'(t)^{2}+\lambda^{2}\ul(t)^{2}
\\
\geq
\min\left\{1,\frac{1}{\mu_{2}}\right\}
\exp\left(-\frac{1}{\mu_{1}}\int_{T_{1}}^{T_{0}}\theta(t)\,dt-2\pi\right)\cdot
\exp\left(\frac{\nu_{2}}{8\mu_{2}}\int_{a}^{b}\theta(\tau)\,d\tau\right)
\label{th:lt-ul}
\end{multline}
for every $t\in[b,T_{0}]$.

\end{enumerate}

\end{prop}

\begin{proof}

We follow the same path as in the case of Proposition~\ref{prop:omega}. To begin with, from the monotonicity of $\theta$ and the second assumption in (\ref{hp:lt-theta}) we obtain that
\begin{equation}
\ep_{i}=
\nu_{2}\,\theta(t_{i})\leq
\nu_{2}\,\theta(a)\leq
\nu_{2}\,\lambda\,\omega\left(\frac{1}{\lambda}\right)
\qquad
\forall i\in\{1,\ldots,k\},
\label{est:ep-omega}
\end{equation}
and therefore from the first inequality in (\ref{hp:lt-lambda}) we deduce that $\ep_{i}\leq 8\gamma^{3}\lambda$, and in particular the assumptions of Lemma~\ref{lemma:altalena} are satisfied in every interval $[t_{i-1},t_{i}]$.

\paragraph{\textmd{\textit{Statement~(1)}}}

Let us start by proving (\ref{th:lt-cl-c*}). To this end, we can assume that $t\in[a,b]$, because otherwise  $\cl(t)$ and $c_{*}(t)$ coincide. When $t\in[t_{i-1},t_{i}]$ for some $i\in\{1,\ldots,k\}$, from the first estimate in (\ref{th:est-varphi}) we obtain that
\begin{equation}
|\cl(t)-c_{*}(t)|=
|\varphi_{\ep_{i},\gamma,\lambda}(t)|\leq
\frac{\ep_{i}}{2\gamma\lambda}%=
\label{est:cl-prelim}
\end{equation}

Plugging (\ref{est:ep-omega}) into (\ref{est:cl-prelim}) we obtain (\ref{th:lt-cl-c*}).

Now let us prove that $c_{\lambda}\in\PS(T_{0},\mu_{1},\mu_{2},\omega,\theta)$. As for the strict hyperbolicity condition, we can limit ourselves to the interval $[a,b]$, where it follows from (\ref{th:lt-cl-c*}) because of the second condition in (\ref{hp:lt-lambda}).

As for the estimate on the derivative, again we can limit ourselves to the interval $[a,b]$. When $t\in[t_{i-1},t_{i}]$ for some $i\in\{1,\ldots,k\}$ we apply the second estimate in (\ref{th:est-varphi}) and we obtain that
\begin{equation}
|\cl'(t)|=
|\varphi_{\ep_{i},\gamma,\lambda}'(t)|\leq
\ep_{i}=
\nu_{2}\,\theta(t_{i})\leq
\theta(t_{i})\leq
\theta(t),
\nonumber
\end{equation}
where the last inequality follows from the monotonicity of $\theta(t)$.

Finally, let us check the $\omega$-continuity of $\cl(t)$. To this end, we consider $t$ and $s$ in $[0,T_{0}]$, we assume without loss of generality that $s<t$, and we distinguish some cases according to the position of $t$ and $s$.
\begin{itemize}

\item  If $t_{i-1}\leq s<t\leq t_{i}$ for some $i\in\{1,\ldots,k\}$, then from (\ref{th:varphi-omega}) we obtain that
\begin{equation}
|\cl(t)-\cl(s)|=
|\varphi_{\ep_{i},\gamma,\lambda}(t)-\varphi_{\ep_{i},\gamma,\lambda}(s)|\leq
\ep_{i}\cdot\max\left\{1,\frac{\pi}{\gamma}\right\}\cdot
\left[\lambda\,\omega\left(\frac{1}{\lambda}\right)\right]^{-1}
\cdot\omega(t-s).
\nonumber
\end{equation}

At this point we exploit (\ref{est:ep-omega}) and our definition (\ref{defn:lt-delta}) of $\nu_{2}$, and we conclude that
\begin{equation}
|\cl(t)-\cl(s)|\leq
\nu_{2}\max\left\{1,\frac{\pi}{\gamma}\right\}\cdot\omega(t-s)\leq
\frac{1}{2}\,\omega(t-s).
\nonumber
\end{equation}

\item  If $s\in[t_{i-1},t_{i}]$ and $t\in[t_{j-1},t_{j}]$ for some $1\leq i<j\leq k$, then from the previous case (and thanks to the factor 1/2) we deduce that
\begin{eqnarray*}
|\cl(t)-\cl(s)| & \leq &
|\cl(t)-\gamma^{2}|+|\gamma^{2}-\cl(s)|
\\[0.5ex]
& = &
|\cl(t)-\cl(t_{j-1})|+|\cl(t_{i})-\cl(s)|
\\[0.5ex]
& \leq &
\frac{1}{2}\,\omega(t-t_{j-1})+\frac{1}{2}\,\omega(t_{i}-s)
\\[0.5ex]
& \leq &
\omega(t-s).
\end{eqnarray*}

\item  All other possibilities for $s$ and $t$ can be dealt with as in the case of the $\omega$-construction.

\end{itemize} 

\paragraph{\textmd{\textit{Statement~(2)}}}

Let us examine the solution to problem (\ref{eqn:uc-lambda})--(\ref{data:ode-lambda}). As in the case of the $\omega$-construction, in the interval $[0,a]$ we have an explicit formula for the solution, from which we deduce that $\ul(a)=0$ and $\ul'(a)=1$. Then in the interval $[t_{0},t_{1}]$ the solution is given by the explicit formula $\ul(t)=w_{\ep_{1},\gamma,\lambda}(t)$, where  $w_{\ep,\gamma,\lambda}$ is defined by (\ref{defn:w}). Since $\gamma\lambda t_{1}$ is an integer multiple of $2\pi$, from the explicit formula we obtain that
\begin{equation}
\ul(t_{1})=0,
\qquad
\ul'(t_{1})=\exp\left(\frac{\ep_{i}}{16\gamma^{2}}(t_{1}-t_{0})\right)=
\exp\left(\frac{\nu_{2}}{16\gamma^{2}}\,\theta(t_{1})(t_{1}-t_{0})\right).
\nonumber
\end{equation}

In the interval $[t_{1},t_{2}]$ the solution is given by the explicit formula $\ul(t)=\alpha w_{\ep_{2},\gamma,\lambda}(t)$, with $\alpha=\ul'(t_{1})$, and therefore
\begin{equation}
\ul(t_{2})=0,
\qquad
\ul'(t_{2})=
\exp\left(\frac{\nu_{2}}{16\gamma^{2}}\,\theta(t_{1})(t_{1}-t_{0})+
\frac{\nu_{2}}{16\gamma^{2}}\,\theta(t_{2})(t_{2}-t_{1})\right).
\nonumber
\end{equation}

At this point by finite induction we find that
\begin{equation}
\ul(b)=\ul(t_{k})=0,
\qquad
\ul'(b)=\ul'(t_{k})=
\exp\left(\frac{\nu_{2}}{16\gamma^{2}}\sum_{i=1}^{k}\theta(t_{i})(t_{i}-t_{i-1})\right).
\label{est:w-b}
\end{equation}

Since $t_{i}-t_{i-1}$ does not depend on $i$, from the monotonicity of $\theta(t)$ we deduce that
\begin{multline}
\quad
\sum_{i=1}^{k}\theta(t_{i})(t_{i}-t_{i-1})=
\frac{2\pi}{\gamma\lambda}\sum_{i=1}^{k}\theta(t_{i})\geq
\frac{2\pi}{\gamma\lambda}\sum_{i=1}^{k-1}\theta(t_{i})\geq
\int_{t_{1}}^{b}\theta(t)\,dt
\\[1ex]
=\int_{a}^{b}\theta(t)\,dt-\int_{a}^{t_{1}}\theta(t)\,dt\geq
\int_{a}^{b}\theta(t)\,dt-\frac{2\pi}{\gamma\lambda}\,\theta(a).
\quad
\nonumber
\end{multline}

Recalling the second condition in (\ref{hp:lt-theta}), and the first condition in (\ref{hp:lt-lambda}), we obtain that
\begin{equation}
\sum_{i=1}^{k}\theta(t_{i})(t_{i}-t_{i-1})\geq
\int_{a}^{b}\theta(t)\,dt-\frac{2\pi}{\gamma}\,\omega\left(\frac{1}{\lambda}\right)\geq
\int_{a}^{b}\theta(t)\,dt-\frac{16\pi\gamma^{2}}{\nu_{2}}.
\nonumber
\end{equation}

Plugging this estimate into (\ref{est:w-b}) we conclude that
\begin{equation}
\ul'(b)^{2}+\gamma^{2}\lambda^{2}\ul(b)^{2}\geq
\exp\left(\frac{\nu_{2}}{8\mu_{2}}\int_{a}^{b}\theta(t)\,dt-2\pi\right).
\nonumber
\end{equation}

At this point in the interval $[b,T_{0}]$ we consider the hyperbolic energy as in the case of the $\omega$-construction and we obtain (\ref{th:lt-ul}).
\end{proof}

%\clearpage

\begin{prop}[Asymptotic activators for initially constant coefficients]\label{prop:final}

Let us assume that $c_{*}\in\mathcal{D}(T_{0},\mu_{1},\mu_{2},\omega,\theta;T_{1},\gamma,\eta)$ for some admissible values of the parameters, and let us define $m(\lambda)$ as in (\ref{defn:m-lambda}), and $M_{3}$ as in (\ref{defn:M12}). 

Let us assume that $m(\lambda)\to +\infty$ as $\lambda\to +\infty$.

Then there exists a family  of asymptotic activators $\{\cl(t)\}\subseteq\PS(T_{0},\mu_{1},\mu_{2},\omega,\theta)$ with rate $\phi(\lambda):=M_{3}\,m(\lambda)$ such that $\cl(t)\to c_{*}(t)$ uniformly in $[0,T_{0}]$.

\end{prop}

\begin{proof}

The strategy is the following. For every $\lambda$ large enough we define a coefficient $\cl\in\PS(T_{0},\mu_{1},\mu_{2},\omega,\theta)$ by modifying $c_{*}$ in some interval $(a_{\lambda},b_{\lambda})$ according to the constructions described in Proposition~\ref{prop:omega} and Proposition~\ref{prop:theta}. We show that $b_{\lambda}\to 0$, and that for $\lambda$ large enough it turns out that
\begin{equation}
|\cl(t)-c_{*}(t)|\leq\frac{\nu_{2}}{\gamma}\,\omega\left(\frac{1}{\lambda}\right)
\qquad
\forall t\in[0,T_{0}],
\label{th:aa-cl-c*}
\end{equation}
where $\nu_{2}$ is defined by (\ref{defn:lt-delta}), and the solutions $\ul(t)$ to problem (\ref{eqn:uc-lambda})--(\ref{data:ode-lambda}) satisfy
\begin{equation}
\ul'(t)^{2}+\lambda^{2}\ul(t)^{2}\geq
M_{4}\exp(2M_{3}\,m(\lambda))
\qquad
\forall t\in[\bl,T_{0}],
\label{th:aa-ul}
\end{equation}
where 
\begin{equation}
M_{4}:=\min\left\{1,\frac{1}{\mu_{2}}\right\}
\exp\left(-\frac{1}{\mu_{1}}\int_{T_{1}}^{T_{0}}\theta(t)\,dt-2\pi
-\frac{\omega(1)}{8\mu_{2}}\right).
\nonumber
\end{equation}

If we prove these claims, then from (\ref{th:aa-cl-c*}) it follows that $\cl\to c_{*}$ uniformly in $[0,T_{0}]$, while (\ref{th:aa-ul}) and the fact that $b_{\lambda}\to 0^{+}$ imply that $\{\cl(t)\}$ is a family of asymptotic activators with rate $\phi(\lambda):=M_{3}\,m(\lambda)$.

In order to define $\cl$ we distinguish two cases. To begin with, we observe that $\omega(\sigma)$ and $\theta(t)$ satisfy (\ref{hp:cor-int-infty}), because otherwise $m(\lambda)$ would be bounded independently of $\lambda$.  

Let $\psi(s)$ denote the function whose minimum is $m(\lambda)$. Due to the second condition in (\ref{hp:cor-int-infty}), the minimum is never attained in $s=0$. Moreover, since $\psi'(s)=\lambda\,\omega(1/\lambda)-\theta(s)$, from the first condition in (\ref{hp:cor-int-infty}) we deduce that $\psi'(T_{0})>0$ when $\lambda$ is large enough, and therefore for these values of $\lambda$ the minimum is not attained also in $s=T_{0}$. Therefore, when $\lambda$ is large enough the minimum is attained in some point $s_{\lambda}\in(0,T_{0})$ where $\psi'(s_{\lambda})=0$, and hence 
\begin{equation}
\theta(s_{\lambda})=\lambda\,\omega\left(\frac{1}{\lambda}\right).
\nonumber
\end{equation}

Exploiting again the first condition in (\ref{hp:cor-int-infty}) we deduce that the right-hand side tends to $+\infty$, and hence $s_{\lambda}\to 0^{+}$. Now let $\Lambda_{\omega}$ denote the set of all $\lambda>0$ such that
\begin{equation}
\lambda\,\omega\left(\frac{1}{\lambda}\right)s_{\lambda}\geq\frac{1}{2}\,m(\lambda),
\label{defn:Lambda-omega}
\end{equation}
and let $\Lambda_{\theta}$ denote the set of remaining $\lambda$'s, for which necessarily it turns out that
\begin{equation}
\int_{s_{\lambda}}^{T_{0}}\theta(t)\,dt\geq\frac{1}{2}\,m(\lambda).
\nonumber
\end{equation}

We are now ready to define $\cl(t)$ in the two cases.

\paragraph{\textmd{\textit{Case $\lambda\in\Lambda_{\omega}$}}}

For every $\lambda\in\Lambda_{\omega}$ we set (here $\lfloor\alpha\rfloor$ denotes the greatest integer less then or equal to $\alpha$)
\begin{equation}
a=a_{\lambda}:=
\frac{2\pi}{\gamma\lambda}\left\lfloor\frac{\gamma\lambda}{4\pi}\,s_{\lambda}-2\right\rfloor,
\qquad\qquad
b=b_{\lambda}:=
\frac{4\pi}{\gamma\lambda}\left\lfloor\frac{\gamma\lambda}{4\pi}\,s_{\lambda}\right\rfloor,
\nonumber
\end{equation}
and we define the coefficient $\cl(t)$ according to the $\omega$-construction of Proposition~\ref{prop:omega}.  Let us check that the assumptions are satisfied if $\lambda\in\Lambda_{\omega}$ is large enough. To begin with, we observe that $a_{\lambda}>0$ because from (\ref{defn:Lambda-omega}) we deduce that $\lambda s_{\lambda}\to +\infty$ when $\lambda\to +\infty$ remaining inside $\Lambda_{\omega}$. In addition, it turns out that
\begin{equation}
0<a_{\lambda}<b_{\lambda}\leq s_{\lambda},
\nonumber
\end{equation}
so that in particular $b_{\lambda}\to 0^{+}$ and $(a_{\lambda},b_{\lambda})\subseteq(0,T_{1})$ for large $\lambda$. Moreover, (\ref{hp:ab-int}) is almost trivial from the definition, while the two inequalities in (\ref{hp:lo-lambda}) are again satisfied when $\lambda$ is large because $\omega(\sigma)\to 0$ as $\sigma\to 0^{+}$. Finally, the first inequality in (\ref{hp:lo-theta}) is true for $\lambda$ large because $b_{\lambda}\to 0$, while the second inequality in (\ref{hp:lo-theta})
is satisfied because
\begin{equation}
\theta(b_{\lambda})\geq
\theta(s_{\lambda})=\lambda\,\omega\left(\frac{1}{\lambda}\right)\geq
\nu_{1}\lambda\,\omega\left(\frac{1}{\lambda}\right).
\nonumber
\end{equation}

At this point we can use the conclusions of Proposition~\ref{prop:omega}. From (\ref{th:lo-cl-c*}) we obtain immediately (\ref{th:aa-cl-c*}) in this case (note that $\nu_{1}=2\nu_{2}$). Finally, we observe that
\begin{equation}
b_{\lambda}-a_{\lambda}\geq\frac{s_{\lambda}}{2},
\nonumber
\end{equation}
and therefore
\begin{equation}
\lambda\,\omega\left(\frac{1}{\lambda}\right)(b_{\lambda}-a_{\lambda})\geq
\frac{1}{2}\lambda\,\omega\left(\frac{1}{\lambda}\right)s_{\lambda}\geq
\frac{1}{4}\,m(\lambda),
\nonumber
\end{equation}
so that (\ref{th:lo-ul}) implies (\ref{th:aa-ul}) in this case.

\paragraph{\textmd{\textit{Case $\lambda\in\Lambda_{\theta}$}}}

For every $\lambda\in\Lambda_{\theta}$ we define $\widehat{s}_{\lambda}$ in such a way that
\begin{equation}
\int_{s_{\lambda}}^{\widehat{s}_{\lambda}}\theta(t)\,dt=
\frac{1}{2}\int_{s_{\lambda}}^{T_{0}}\theta(t)\,dt\geq\frac{1}{4}m(\lambda),
\nonumber
\end{equation}
and we observe that $\widehat{s}_{\lambda}\to 0^{+}$ because the integral of $\theta(t)$ in $(0,T_{1})$ is divergent. Then we set (here $\lceil\alpha\rceil$ denotes the smallest integer greater then or equal to $\alpha$)
\begin{equation}
a=a_{\lambda}:=
\frac{2\pi}{\gamma\lambda}\left\lceil\frac{\gamma\lambda}{2\pi}\,s_{\lambda}\right\rceil,
\qquad\qquad
b=b_{\lambda}:=
\frac{2\pi}{\gamma\lambda}\left\lceil\frac{\gamma\lambda}{2\pi}\,\widehat{s}_{\lambda}\right\rceil,
\nonumber
\end{equation}
and we define the coefficient $\cl(t)$ according to the $\theta$-construction of Proposition~\ref{prop:theta}.

With these notations it turns out that
\begin{equation}
0<
s_{\lambda}\leq 
a_{\lambda}\leq
b_{\lambda}
\qquad\text{and}\qquad
\widehat{s}_{\lambda}\leq
b_{\lambda}\leq
\widehat{s}_{\lambda}+\frac{2\pi}{\gamma\lambda},
\nonumber
\end{equation}
so that in particular $b_{\lambda}\to 0^{+}$ and $(a_{\lambda},b_{\lambda})\subseteq(0,T_{1})$ for large $\lambda$. Moreover, (\ref{hp:ab-int}) is almost trivial from the definition, while the two inequalities in (\ref{hp:lt-lambda}) are again satisfied when $\lambda$ is large because $\omega(\sigma)\to 0$ as $\sigma\to 0^{+}$. Finally, the first inequality in (\ref{hp:lt-theta}) is true for $\lambda$ large because $b_{\lambda}\to 0$, while the second inequality in (\ref{hp:lt-theta}) is satisfied because
\begin{equation}
\theta(a_{\lambda})\leq
\theta(s_{\lambda})=
\lambda\,\omega\left(\frac{1}{\lambda}\right).
\nonumber
\end{equation}

At this point we can use the conclusions of Proposition~\ref{prop:theta}. From (\ref{th:lt-cl-c*}) we obtain immediately (\ref{th:aa-cl-c*}) in this case. Finally, we observe that
\begin{multline}
\quad
\int_{a_{\lambda}}^{b_{\lambda}}\theta(t)\,dt\geq
\int_{s_{\lambda}+2\pi/(\gamma\lambda)}^{\widehat{s}_{\lambda}}\theta(t)\,dt=
\int_{s_{\lambda}}^{\widehat{s}_{\lambda}}\theta(t)\,dt-
\int_{s_{\lambda}}^{s_{\lambda}+2\pi/(\gamma\lambda)}\theta(t)\,dt
\\[1ex]
\geq\frac{1}{4}m(\lambda)-
\frac{2\pi}{\gamma\lambda}\,\theta(s_{\lambda})\geq
\frac{1}{4}\,m(\lambda)-\frac{2\pi}{\gamma}\,\omega\left(\frac{1}{\lambda}\right),
\quad
\nonumber
\end{multline}
and therefore
\begin{equation}
\int_{a_{\lambda}}^{b_{\lambda}}\theta(t)\,dt\geq
\frac{1}{4}\,m(\lambda)-\frac{2\pi\omega(1)}{\gamma}
\nonumber
\end{equation}
when $\lambda$ is large enough. Recalling that $\nu_{2}\leq\gamma/(2\pi)$, at this point (\ref{th:lt-ul}) implies (\ref{th:aa-ul}) in this case.
\end{proof}

%\clearpage

\setcounter{equation}{0}
\section{Proof of Corollaries}\label{sec:cor}

\paragraph{\textmd{\textit{Proof of Corollary~\ref{cor:int-infty}}}}

It is enough to show that $m(\lambda)\to +\infty$. Let $\lambda_{n}\to +\infty$ be any sequence of positive real numbers. For every positive integer $n$, let us choose 
\begin{equation}
s_{n}\in\operatorname{argmin}\left\{\lambda_{n}\,\omega\left(\frac{1}{\lambda_{n}}\right)s+
\int_{s}^{T_{0}}\theta(t)\,dt:s\in[0,T_{0}]\right\}.
\nonumber
\end{equation}

Up to subsequences (not relabeled) we can assume that $s_{n}\to s_{\infty}\in[0,T_{0}]$. If $s_{\infty}=0$, then $m(\lambda_{n})\to +\infty$ due to the second term in the minimum and the second assumption in (\ref{hp:cor-int-infty}). If $s_{\infty}>0$, then $m(\lambda_{n})\to +\infty$ due to the first term in the minimum and the first assumption in (\ref{hp:cor-int-infty}).
\hfill$\Box$

\paragraph{\textmd{\textit{Proof of Corollary~\ref{cor:1/t} -- Statement (1)}}}

From assumption (\ref{hp:cor-1/t-theta-1}) we deduce that there exists a positive real number $M$ such that $c(t)\leq M/t$ for every $t\in(0,T_{0}]$, and hence
\begin{equation}
\lambda\,\omega\left(\frac{1}{\lambda}\right)s+\int_{s}^{T_{0}}\theta(t)\,dt\leq
\lambda\,\omega\left(\frac{1}{\lambda}\right)s+M(\log T_{0}-\log s).
\nonumber
\end{equation}

In particular, setting $s=\lambda^{\alpha-1}$ we obtain that
\begin{equation}
m(\lambda)\leq
\lambda^{\alpha}\omega\left(\frac{1}{\lambda}\right)+M(\log T_{0}+(1-\alpha)\log\lambda).
\nonumber
\end{equation}

Now we divide by $\log\lambda$ and we let $\lambda\to+\infty$. Due to assumption (\ref{hp:cor-1/t-omega-1}) we deduce that
\begin{equation}
\limsup_{n\to+\infty}\frac{m(\lambda)}{\log\lambda}\leq M(1-\alpha)
\nonumber
\end{equation}
for every $\alpha\in(0,1)$. Finally, letting $\alpha\to 1^{-}$ we conclude that actually $m(\lambda)/\log\lambda\to 0$, which implies well-posedness with at most an arbitrarily small derivative loss.

\paragraph{\textmd{\textit{Proof of Corollary~\ref{cor:1/t} -- Statement (2)}}}

From assumptions (\ref{hp:cor-1/t-omega-2}) and (\ref{hp:cor-1/t-theta-2}) we deduce that there exists positive real numbers $m_{2}$ and $m_{1}$ such that
\begin{equation}
\omega(\sigma)\geq m_{1}\,\sigma^{\alpha}
\qquad\quad\text{and}\quad\qquad
\theta(t)\geq\frac{m_{2}}{t},
\nonumber
\end{equation}
and in particular
\begin{equation}
\lambda\,\omega\left(\frac{1}{\lambda}\right)s+\int_{s}^{T_{0}}\theta(t)\,dt\geq
m_{1}\lambda^{1-\alpha}s+m_{2}(\log T_{0}-\log s).
\nonumber
\end{equation}

Minimizing the right-hand side with respect to $s$ we conclude that
\begin{equation}
m(\lambda)\geq m_{2}\left(1+\log T_{0}+\log(m_{1}\lambda^{1-\alpha})-\log m_{2}\right),
\nonumber
\end{equation}
and hence
\begin{equation}
\liminf_{\lambda\to +\infty}\frac{m(\lambda)}{\log\lambda}\geq m_{2}(1-\alpha)>0,
\nonumber
\end{equation}
which implies that the derivative loss is actually finite, and proportional to $1-\alpha$, for a residual set of coefficients.
\hfill$\Box$

%\clearpage

\subsubsection*{\centering Acknowledgments}

Both authors are members of the Italian \selectlanguage{italian}``Gruppo Nazionale per l'Analisi Matematica, la Probabilit\`{a} e le loro Applicazioni'' (GNAMPA) of the ``Istituto Nazionale di Alta Matematica'' (INdAM). 

\selectlanguage{english}

%%\clearpage

%%{\small 
%\bibliographystyle{../../../BibTeX/MaxNew}
%\bibliography{../../../BibTeX/DGCS}
%%}

\label{NumeroPagine}

\end{document}